\documentclass{amsart}%
\usepackage{cite}
\usepackage[hypertex]{hyperref}
\usepackage{amssymb}
\usepackage{amsmath}
\usepackage{amsthm}
\usepackage{amsfonts}
\usepackage{graphicx}
\usepackage{color}
\usepackage[toc,page]{appendix}%

\newtheorem{theorem}{Theorem}[section]

\newtheorem{notation}[theorem]{Notation}
\newtheorem{prop}[theorem]{Proposition}

\newtheorem{as}[theorem]{Assumption}
\newtheorem{remark}[theorem]{Remark}
\theoremstyle{definition}

\newtheorem{definition}[theorem]{Definition}

\newtheorem{example}{Example}[section]
\numberwithin{equation}{section}

\title[Hypoelliptic heat kernels on nilpotent Lie groups]{Hypoelliptic heat kernels on nilpotent Lie groups}
\author[Asaad]{Malva Asaad{$^{\dagger}$}}
\thanks{\footnotemark {$\dagger$} This research was supported in part by NSF Grant DMS-1007496.}
\address{Department of Mathematics\\
University of Connecticut\\
Storrs, CT 06269, USA} \email{malva.asaad@uconn.edu}
\author[Gordina]{Maria Gordina{$^{\dagger}$}}
\thanks{\footnotemark {$\dagger$} This research was supported in part by NSF Grant DMS-1007496.}
\address{Department of Mathematics\\
University of Connecticut\\
Storrs, CT 06269, USA} \email{maria.gordina@uconn.edu}

\keywords{
nilpotent group, sub-Riemannian manifold, hypoelliptic heat kernel, Kirillov's orbit method}

\subjclass{Primary 58J35; Secondary 53C17}

\begin{document}

\begin{abstract}
The starting point of our analysis is an old idea of writing an eigenfunction expansion for a heat kernel considered in the case of a hypoelliptic heat kernel on a nilpotent Lie group $G$. One of the ingredients of this approach is the generalized Fourier transform. The formula one gets using this approach is explicit as long as we can find all unitary irreducible representations of $G$. In the current paper we consider an $n$-step nilpotent Lie group $G_{n}$ as an illustration of this techinique. First we apply Kirillov's orbit method to describe these representations for $G_{n}$. This allows us to write the corresponding hypoelliptic heat kernel using an integral formula over a Euclidean space. As an application, we describe a short-time behaviour of the hypoelliptic heat kernel in our case.
\end{abstract}

\maketitle

\tableofcontents

\section{Introduction}\label{s.1}

The subject of this paper is a non-commutative variation on the classical idea of finding an eigenfunction expansion of the heat kernel. We refer to \cite{Grigoryan1999b} for an excellent review of the subject and its connections to heat kernel estimates on Riemannian manifolds. The heat kernel we consider is the hypoelliptic heat kernel for a sub-Laplacian on a nilpotent Lie group equipped with a natural sub-Riemannian structure. Recall that this setting provides a classical model for approximating of hypoelliptic operators first introduced by L.~Rothschild and E.~Stein in \cite{RothschildStein1976}.

Our goal is to describe such a hypoelliptic heat kernel in a way that allows to study further its properties such as short and long time behavior of such kernels. An eigenfunction expansion for the hypoelliptic heat kernel on $\operatorname{SU}\left( 2 \right)$ has been studied in \cite{BaudoinBonnefont2009}, and a more abstract expansion has been presented in \cite{AgrachevBoscainGauthierRossi2009}. The latter is based on a generalized Fourier transform (GFT) which informally speaking can be viewed as a generalized eigenfunction expansion. This requires a complicated machinery such as direct integrals, therefore applying it to an object such as an unbounded operator is not possible in general. This fundamental difficulty has not been addressed in \cite{AgrachevBoscainGauthierRossi2009}, but we make it precise in the case of a connected simply connected nilpotent Lie group $G$. We also want to point out that even though there are choices to be made for an appropriate measure on a general sub-Riemannian manifold, the use of the GFT forces one to use a Haar measure. This complicated interplay between algebraic and sub-Riemannian structure in the case of Lie groups still needs more study.

The main ingredient in our analysis is Kirillov's orbit method. This technique allows us to describe explicitly the unitary dual of $G$, that is, the space of equivalency classes of irreducible unitary representations of $G$. We address the domain issue in Theorem \ref{t.4.10} which does not assume that $G$ is nilpotent. But the description using the orbit method reduces this abstract description to function spaces over a Euclidean space as can be seen in the proof of Theorem \ref{t.4.11}. This is our main result, and the explicit formula for the hypoelliptic heat kernel on a nilpotent group in this theorem is the result of applying the GFT to the hypoelliptic Laplacian on $G$, and then using the inversion formula for this transform. In particular, this formula uses a heat kernel for a certain differential operator on a Euclidean space. This operator is a second order differential operator with  polynomial coefficients and a polynomial potential. We present the ingredients of both Kirillov's orbit method and this Schr\"{o}dinger-like operator in the case of the Heisenberg group and an $n$-step nilpotent group. In these cases this operator is the Schr\"{o}dinger operator with a  quadratic potential on a certain $\mathbb{R}^{N}$. For related results we refer to \cite{HulanickiJenkins1984, HulanickiJenkins1987}. The dimension of this Euclidean space and the degree of the potential depend on the structure of $G$ in terms of the orbit method. This connection with the heat kernel for such operators provides a way to prove bounds and functional inequalities for the hypoelliptic heat kernel. We only mention a simple estimate in the paper, but we expect that this approach will be used in the future. Moreover, it would be interesting to see if the short time estimates one gets from this analysis can be formulated in terms of the sub-Riemannian distance. Our main new example is an $n$-step nilpotent group, though we also mention previously studied examples, such as the Heisenberg group.

The paper is organized as follows. In Section \ref{s2} we review Kirillov's orbit method for nilpotent groups and illustrate it by  describing explicitly the irreducible representations of the $n$-step nilpotent Lie group of the growth vector $(2, 3, ... , n+1)$, $G_{n+1}$. This turns out to be an inductive generalization of the computations carried out in \cite{KirillovOrbitMethodAMSBook} for the Heisenberg group. In Section \ref{s3}, we recall basic definitions from sub-Riemannian geometry, in particular for left-invariant structures on Lie groups. In Section \ref{s4}, we use the generalized Fourier transform and its properties to describe the hypo-elliptic heat kernel on a nilpotent group. Finally in Section \ref{s4}, we  apply this technique to the group $G_{n+1}$. In Section \ref{s6} we use the formula for the hypo-elliptic kernel to study the short-time behaviour of heat kernels. In that we are motivated by the approach introduced by S\'eguin and Mansouri \cite{SeguinMansouri2012}. Note that we avoid introducing assumptions as in \cite[p.3904]{SeguinMansouri2012}, as we can use the explicit expressions we derive using the orbit method, and thus avoid ambiguities in their assumptions.

There are a number of papers related to the subject of this paper, and below we mention just a few. Besides the ones we already referred to earlier such as recent results in \cite{AgrachevBoscainGauthierRossi2009, BaudoinBonnefont2009, SeguinMansouri2012}, we would like to mention earlier explicit hypoelliptic heat kernel expressions in \cite{Gaveau1977a, Hulanicki1976a, BealsGaveauGreiner2000}. Our approach is different as we use Kirillov's orbit method to represent a hypoelliptic heat kernel on a nilpotent group by an integral formula over a Euclidean space. This is what makes it different from using analysis on the group itself as is done in  \cite{HelfferNourrigat1979, DixmierMalliavin1978}.  We leave out a vast literature on various estimates on hypoelliptic heat kernels.

\section{Kirillov's orbit method}\label{s2}

Kirillov's theory gives an explicit description of irreducible unitary representations of a nilpotent Lie group. Throughout this paper we assume that $G$ is connected and simply connected. The main ingredient of Kirillov's theory is the set of the orbits of the coadjoint action. Namely, the coadjoint map allows to describe all irreducible unitary representations of $G$: the set of equivalency classes of all irreducible unitary representations of $G$ are naturally parameterized by the orbits of $\mathfrak{g}^{\ast}$ under the coadjoint action. In this section, we first review  Kirillov's orbit theory \cite{KirillovOrbitMethodAMSBook}, \cite[Section 3.1]{CorwinGreenleafBook}, and then apply it to an $n$-step nilpotent Lie group $G_{n+1}$  which allows us to describe explicitly all irreducible representations of $G_{n+1}$.

Recall that a Lie group $G$ acts on its Lie algebra $\mathfrak{g}$ by the \emph{adjoint action}

\[
\operatorname{Ad}_{g}X:=gXg^{-1}, g \in X \in \mathfrak{g},
\]
and $G$ acts on the dual space $\mathfrak{g}^{\ast}$ by the \emph{coadjoint action} as follows

\[
Ad^{\ast}_{g}l\left( X \right):=l\left( g^{-1}Xg\right), g \in G, l \in \mathfrak{g}^{\ast}, X \in \mathfrak{g}.
\]
Finally, the orbit of $l$ in $\mathfrak{g}^{\ast}$ under the coadjoint action $Ad^{\ast}$ is denoted by  $\mathcal{O}_{l}$.

\begin{definition}
For an element $l\in \mathfrak{g}^*$ we define its  \emph{radical} by
\[
rad_l:=\left\{ Y\in\mathfrak{g}: l[X,Y]=0 \text{ for all }\  X \in \mathfrak{g} \right\}.
\]
\end{definition}
An equivalent description of the radical is given by

\[
rad_l=\left\{ Y\in\mathfrak{g}: \left(\left( \operatorname{ad}_{Y}\right)^{\ast}l\right)\left( X \right)=0 \text{ for all }\  X \in \mathfrak{g} \right\},
\]
and $rad_l$ is a subalgebra of $\mathfrak{g}$.

The next step in Kirillov's method is to find a polarizing subalgebra.

\begin{definition} An ideal $\mathfrak{m}\subset \mathfrak{g}$ is called a \emph{polarizing (or maximal subordinate) subalgebra} for $l$ if it satisfies
\begin{enumerate}
\item $rad_l\subset \mathfrak{m}$;
\item $\operatorname{dim} \mathfrak{m}=\frac{\operatorname{dim} rad_l+\operatorname{dim} \mathfrak{g}}{2}$;
\item $l\left( [\mathfrak{m}, \mathfrak{m}] \right)=0$.
\end{enumerate}
\end{definition}
The last property means that this subalgebra is subordinate for $l$, and the second one indicates that the subalgebra is of maximal dimension. The first property is actually a result proven first by Puk{\'a}nszky in \cite[p. 157]{PukanszkyBook1967}. Recall that for a nilpotent Lie algebra such a subalgebra always exists by \cite[Theorem 1.3.3 ]{CorwinGreenleafBook}.

\begin{remark} As \cite[p.30]{CorwinGreenleafBook} says for each $l \in \mathfrak{g}^{\ast}$ its radical is uniquely determined, but there might be more than one polarizing subalgebra for $l$. When one wants to avoid this ambiguity, one can use Vergne's construction in \cite{Vergne1970a, Vergne1970b} which gives a canonical way to choose a polarizing subalgebra for $l \in \mathfrak{g}^{\ast}$  and a strong Mal'cev basis in $\mathfrak{g}$.
\end{remark}

\subsection{An $n$-step nilpotent Lie group $G_{n+1}$}\label{ss2.1} We would like to illustrate how to find these objects by considering an $n$-step nilpotent Lie group $G_{n+1}$ also known as a thread-like group. Let  $\mathfrak{g}_{n+1}$  be an $(n+1)$-dimensional Lie algebra  generated by $X, Y_1,\dots, Y_n$ satisfying
\[
[X,Y_i]= Y_{i+1}, [X,Y_n]=0,\ \text{and}\ [Y_i,Y_j]=0\ \text{ for all } i,j.
\]
This is an $n$-step nilpotent Lie algebra. A realization of $\mathfrak{g}_{n+1}$  as a matrix algebra is obtained by letting  $W=aX+\sum_{i=1}^nb_iY_i $ corresponding  to the $(n+1)\times (n+1)$ matrix of $ \mathfrak{g}_{n+1}$

\[
\begin{pmatrix}
0 & a & 0 &0& \cdots &0 &  b_n\\
0 & 0 &a &0& \cdots &0 & b_{n-1}\\
0 & 0 &0 & a&\cdots &0 & b_{n-1}\\
\vdots & \vdots & \vdots& \vdots& \ddots&  \vdots&\vdots\\
0 & 0 & 0 &0& \cdots &a&  b_2\\
0 & 0 & 0 &0& \cdots &0 &  b_1\\
0 & 0 & 0 &0& \cdots &0 &  0
\end{pmatrix}.
\]
Let $G_{n+1}$ be a connected, simply connected nilpotent Lie group, with the Lie algebra $ \mathfrak{g}_{n+1}$, then the exponential map $\exp:  \mathfrak{g}_{n+1} \rightarrow G_{n+1}$ is an analytic diffeomorphism. Moreover, if

\[
W=aX+\sum_{i=1}^nb_iY_i \in \mathfrak{g}_{n+1},
\]
then
\[
w=\exp\left(aX+\sum_{i=1}^nb_iY_i\right) \in G_{n+1}
\]
can be written as
\begin{equation}\label{e.3.1}
w=
\left(
  \begin{array}{cccccc}
   1 & a & \frac{a^2}{2!} &\frac{a^3}{3!} & \cdots  & z_n
   \\
   0 & 1 & a &\frac{a^2}{2!} & \cdots & z_{n-1}
   \\
0 & 0 & 1 & a & \cdots & z_{n-2}
\\
\vdots & \vdots & \vdots & \vdots &  \ddots & \vdots
\\
0 & 0 & 0 & 0 & \cdots & z_1
\\
0 & 0 & 0 & 0 & \cdots & 1
  \end{array}
\right),
\end{equation}
where
\begin{equation}\label{e.3.18}
z_j:=\sum_{k=0}^{j-1}\frac{a^k b_{j-k}}{(k+1)!}, j=1, ..., n.
\end{equation}

\subsection{Coadjoint action, coadjoint orbits and polarizing subalgebras for $G_{n}$}

Suppose

\begin{align*}
& W=aX+\sum_{i=1}^nb_iY_i\in \mathfrak{g}_{n+1},
\\
& w=\exp(xX+\sum_{i=1}^ny_iY_i)\in G_{n+1},
\end{align*}
then the adjoint action is given by
\[
Ad_wW=wWw^{-1}=
\left(
  \begin{array}{ccccccc}
    0 & a & 0 &0& \cdots &0 &  \xi_n\\
0 & 0 &a &0& \cdots &0 & \xi_{n-1}\\
0 & 0 &0 & a&\cdots &0 & \xi_{n-1}\\
\vdots & \vdots & \vdots& \vdots& \ddots&  \vdots&\vdots\\
0 & 0 & 0 &0& \cdots &a&  \xi_2\\
0 & 0 & 0 &0& \cdots &0 &  \xi_1\\
0 & 0 & 0 &0& \cdots &0 &  0
  \end{array}
\right),
\]
where
\[
\xi_j=b_j+\sum_{k=1}^{j-1}\frac{x^{k-1}}{k!}\{xb_{j-k}-ay_{j-k}\}.
\]
Finally we can use the basis $\{ X, Y_1,..., Y_n \}$ of $\mathfrak{g}_{n+1}$ and the dual basis $\{X^*,Y_1^*,...,Y_n^*\}$ of $\mathfrak{g}_{n+1}^*$, to see that the co-adjoint action is given by
\begin{eqnarray*}
Ad^*_w(l)(W)=
l(Ad_{w^{-1}}W)&=&a\left(\alpha+\sum_{i=2}^n\beta_i\left(\sum_{k=1}^{i-1}(-1)^{k+1}y_{i-k}\frac{x^{k-1}}{k!}\right)\right)\\
&+&\sum_{j=1}^n\left(b_j\sum_{k=j}^n\beta_k(-1)^{k-j}\frac{x^{k-j}}{(k-j)!}\right),
\end{eqnarray*}
for any $l=\alpha X^*+\sum_{i=1}^n\beta_iY_i^* \in \mathfrak{g}_{n+1}^{\ast}$.

For $m \in \mathbb{N}$ we define the polynomial function

\begin{equation}
f_j(x; \mathbf{B}_{m}):=\sum_{k=j}^{m} (-1)^{k-j}\frac{\beta_k}{\beta_{m}^{k-j}}\frac{(\beta_{m-1}-x)^{k-j}}{(k-j)!},\label{e.2.1}
\end{equation}
where $x \in \mathbb{R}$, $\mathbf{B}_{m}= (\beta_1,\beta_2, ... , \beta_m)
 \in \mathbb{R}^{m}, \beta_{m}\not= 0$.

Observe that the orbit $\mathcal{O}_{l}$ of $l=\alpha X^*+\beta_1Y_1^*+\cdots+\beta_mY_m^*$ with $\beta_m \not= 0$ under the coadjoint action is two-dimensional if $m\geqslant 3$  and is given by
\begin{align*}
&  \mathcal{O}_{l}=\mathcal{O}_{\left( \alpha,\mathbf{B}_{m}, 0, ..., 0 \right)}=
\\
& \left\{yX^*+\sum_{j=1}^{m-2}f_j(x; \mathbf{B}_{m})Y_j^*+xY_{m-1}^*+\beta_{m}Y_{m}^*,  x, y \in \mathbb{R} \right\}.
\end{align*}
Here we identified $l \in \mathfrak{g}^{\ast}$ with $\left( \alpha,\mathbf{B}_{m}, 0, ..., 0 \right)$.
In case of $m=2$ then the orbit of $l=\alpha X^*+\beta_1Y_1^*+\beta_2Y_2^*$ is again two-dimensional and has the following form:
\[
\mathcal{O}_{l}=\mathcal{O}_{\left( \alpha,\beta_1,\beta_2, 0, ..., 0 \right)}=
 \left\{pX^*+qY^*_1+\beta_2Y_{2}^*,  p,q \in \mathbb{R} \right\}.
 \]
and in this case
\[
rad_l=
\{b_2Y_2+\cdots+b_nY_n, b_i\in \mathbb{R}\}
\]
Finally, in the case where $m=1$ then the orbit of $l=\alpha X^*+\beta_1Y_1^*$ is a zero-dimensional orbit (point orbits)
\[  \mathcal{O}_{l}=\mathcal{O}_{\left( \alpha,\beta_1, 0, ..., 0 \right)}= \left\{\alpha X^*+\beta_1Y^*_1 \right\}.
\]
and in this case $rad_l=\mathfrak{g}$.

\begin{prop}\label{p.2.2} For any $l\in \mathfrak{g}_{n+1}^*$ and $m \geqslant 3$
\begin{align*}
& rad_l=
\\
& \left\{\sum_{j=1}^{m-3}b_j\left(Y_j-\frac{f_{j+1}(0;\mathbf{B}_m)}{\beta_m}Y_{m-1}\right) + b_{m-2}Y_{m-2}+b_mY_m+\cdots+b_nY_n
\right\}
\end{align*}
for some $\mathbf{B}_m$, where $f_{j+1}(0;\mathbf{B}_m)$ is defined by \eqref{e.2.1}.
 \end{prop}

\begin{proof}
Note that the definition of the radical is independent of the representative in the orbit, therefore it is enough to compute the radical for elements
\begin{equation}\label{e.2.2}
l=\sum_{j=1}^{m-2}f_j(0;\mathbf{B}_m)Y_j^*+\beta_mY_m^*.
\end{equation}
An arbitrary element of $\mathfrak{g}_{n+1}$ is of the form
\[
Z=aX+\sum_{i=1}^nb_iY_i
\]
and therefore
\begin{align*}
& [X, Z]=b_1Y_2+\cdots +b_{n-1}Y_n,
\\
& [Y_{m-1}, Z]=-aY_m.
\end{align*}
Thus
\[
rad_l=\{Z\in \mathfrak{g}_{n+1}\colon l[X, Z]=l[Y_1, Z]=\cdots =l[Y_n, Z]=0\},
\]
and so we get the relations
\begin{align*}
& \beta_m a=0,
\\
& \beta_3b_2=0 \text{ if } m=3,
\\
& \beta_mb_{m-1}+\sum_{j=2}^{m-2}b_{j-1}f_j(0; \beta_1, ..., \beta_{m})=0 \text{ if } m\geqslant 4.
\end{align*}
Recall that $\beta_m\not= 0$, and therefore it is enough to consider two cases as follows.

\noindent \emph{Case 1:} $m=3$. In this case $a=b_2=0$ and the radical is
\[
rad_l=\operatorname{Span}\left\{ Y_1, Y_3, Y_4, ..., Y_n\right\}.
\]
\emph{Case 2:} $m\geqslant 4$. In this case we get $a=0$ and
\[
b_{m-1}=-\sum_{j=2}^{m-2}\frac{b_{j-1}}{\beta_m}f_j(0; \mathbf{B}_m).
\]

\end{proof}

Now we can describe polarizing algebras for $\mathfrak{g}_{n+1}$ explicitly.

\begin{prop} Suppose $l=\alpha X^*+\beta_1Y_1^*+\cdots+\beta_mY_m^*\in \mathfrak{g}_{n+1}^{\ast}$ with $\beta_m \not= 0$ for some $m\geqslant 3$, then a polarizing algebra for $l$ is unique, and is given by

\[
\mathfrak{m}=\operatorname{Span}\left\{ Y_1, ..., Y_n \right\}.
\]
\end{prop}

\begin{proof}
By Proposition \ref{p.2.2} we see that $\operatorname{dim} rad_l=n-1$ and since $\operatorname{dim} \mathfrak{g}_{n+1}=n+1$, we have $\operatorname{dim} \mathfrak{m}=n$. This means that to find a polarizing subalgebra it is enough to add one element $Y \in \mathfrak{g}_{n+1}$ to $rad_l$ such that $Y$ is linearly independent of $rad_l$ and the ideal generated by $rad_l$ and $Y$ is not the whole algebra
$\mathfrak{g}_{n+1}$. It can be easily seen that such an $\mathfrak{m}$ is unique and it is equal to
\[
\mathfrak{m}=\operatorname{Span}\left\{ Y_1, ..., Y_n \right\}
\]
as long as $m\geqslant 3$ in \eqref{e.2.2}.
\end{proof}
For the case $m=2$ the orbit is again two-dimensional, but the polarizing subalgebra is not unique. For example, the following two ideals are polarizing subalgebras
\begin{align*}
& \mathfrak{m}_1=\operatorname{Span}\left\{ Y_1, Y_2, ..., Y_n \right\},   \text{ or }
\\
& \mathfrak{m}_2=\operatorname{Span}\left\{ X, Y_2, ..., Y_n \right\}.
\end{align*}
Finally, for $m=1$ the zero-dimensional orbits are one--point orbits, namely,
\[
\mathcal{O}_{(\alpha, \beta_1, 0, ..., 0)}=\left\{ l_{\alpha, \beta_1, 0, ..., 0} \right\},
\]
in which case $rad_{l}=\mathfrak{m}=\mathfrak{g}_{n+1}$.

\subsection{Induced representations for nilpotent Lie groups}\label{ss.3.4} To apply Kirillov's orbit method to the representation theory of nilpotent groups we first recall some basic facts about induced representations which can be found in  \cite{CorwinGreenleafBook, FollandHABook} among many other sources.

\begin{definition}\label{d.3.5} The set $\widehat{G}$ of equivalence classes of irreducible unitary representations of a locally compact group $G$  is called the \emph{(unitary) dual space} of $G$.
\end{definition}

If $M$ is a closed connected subgroup of $G$, we can use the induction procedure to describe representations of $G$ via representations of $M$. Suppose $( \chi, \mathcal{H}_\chi )$ is a representation of $M$. Then  $( \chi, \mathcal{H}_\chi )$  yields a natural representation $\pi=\mbox{Ind}(M\uparrow G, \chi)$ of $G$ in a new Hilbert space $\mathcal{H}_\pi$. We start by constructing the representation Hilbert space $\mathcal{H}_\pi$. This space is defined as a Hilbert space of equivalence classes of Borel measurable vector-valued functions $f: G \to \mathcal{H}_\chi$ such that

\begin{align}
& f(m g)=\chi(m)^{-1} f(g), m \in M, g \in G, \notag
\\
& \int\limits_{ G / M}\Vert f(g)\Vert^2 d\mu\left( g \right)< \infty, \label{e.3.2}
\end{align}
where  $\mu\left( g \right)$ is a left-invariant measure on $G / M$.

Note that in our case  both $G$ and $M$ are nilpotent and thus unimodular,  and therefore such a measure always exists on the homogeneous space $G / M$ (e.g. \cite[Lemma 1.2.13]{CorwinGreenleafBook}). In measure--theoretical terms the measure $\mu$ is the pushforward of the Haar measure $dg$ by the quotient map

\begin{align}
& q:G \longrightarrow G / M, \notag
\\
& q\left( g \right):=gM.\label{e.3.6}
\end{align}

For nilpotent groups the induction procedure described above is applied to a particular choice of the representation $\chi$, namely, for a fixed $l \in \mathfrak{g}^{\ast}$ we define the character of $M$ by

\begin{equation}\label{e.3.7}
\chi_{l}\left( \exp\left( X \right) \right):=\exp\left( 2\pi i l\left( X \right)\right), X \in \mathfrak{m},
\end{equation}
where we used the fact that $\exp: \mathfrak{m} \longrightarrow M$ is a (global) diffeomorphism. In this case we choose the Hilbert space $\mathcal{H}_{\pi_{l, \mathfrak{m}}}$ to be  the Hilbert space of equivalence classes of Borel measurable functions $f: G \to \mathbb{C}$  satisfying \eqref{e.3.2}. Note that the map $g\mapsto \Vert f(g)\Vert^2$ is constant on each left coset, so the integral in \eqref{e.3.2} exists. The space $\mathcal{H}_{\pi_{l, \mathfrak{m}}}$ is a completion of the space of such functions with respect to the inner product
\[
\langle f_{1}(g), f_{2}(g)\rangle:=\int_{ G / M}\langle f_{1}(g), f_{2}(g)\rangle d\mu\left( g \right)
\]
which is again well-defined since $\langle f_{1}(g), f_{2}(g)\rangle$ is constant on each left coset.

Finally the induced representation $\pi_{l, \mathfrak{m}}$ is defined by letting $G$ act on the right as follows
\[
\pi_{l, \mathfrak{m}}(x)f(g):=f(x^{-1} g),  \text{ for all }  x\in G, f \in \mathcal{H}_{\pi_{l, \mathfrak{m}}},
\]
which is a unitary operation. This is an irreducible representation of $G$ on $\mathcal{H}_{\pi_{l, \mathfrak{m}}}$, and any irreducible representation of $G$ can be written as an induced representation for some $l \in \mathfrak{g}^{\ast}$ and any polarizing subalgebra $\mathfrak{m}$ for $l$ by \cite[p. 124]{CorwinGreenleafBook}. That is, for any $l \in \mathfrak{g}^{\ast}$ and two polarizing subalgebras these representations are unitarily equivalent.

By  \cite[pp. 124-125]{CorwinGreenleafBook} and \cite[p. 159]{FollandHABook} there is an isometry between $\mathcal{H}_\chi$ and $L^2\left( \mathbb{R}^{k}, dx\right)$ for some $k$ with respect to the Lebesgue measure $dx$. This allows us to find a representation on $L^2\left( \mathbb{R}^{k}, dx\right)$ which is unitarily equivalent to $\pi_{l, \mathfrak{m}}$. Moreover, we can identify smooth vectors $C^{\infty}$ of the representation $\pi_{l, \mathfrak{m}}$ with the Schwartz space on this $\mathbb{R}^{k}$.

To make these isometries more explicit in the nilpotent case we need to choose a weak Mal'cev basis of $\mathfrak{g}$ passing through $\mathfrak{m}$, that is, a basis $\left\{ X_{1}, ..., X_{n}\right\}$ of $\mathfrak{g}$ such that $\left\{ X_{1}, ..., X_{k}\right\}$ is a basis of $\mathfrak{m}$, and $\operatorname{Span}\{X_{1}, ..., X_{j}\}$ is a subalgebra of $\mathfrak{g}$ for any $1\leqslant j \leqslant n$. This basis allows us to have an explicit description of $G / M$ as described by the following theorem which is based on \cite[Theorem 1.2.12 ]{CorwinGreenleafBook}.

\begin{theorem}\label{t.3.5} Let $\mathfrak{m}$ be a $k$-dimensional subalgebra of the nilpotent Lie algebra $\mathfrak{g}$, let $M=\exp(\mathfrak{m})$ and $G=\exp(\mathfrak{g})$, and let  $\{X_1, ..., X_n\}$ be a weak Mal'cev basis for $\mathfrak{g}$ through $\mathfrak{m}$. Define $\phi: \mathbb{R}^{n-k}\rightarrow G / M$ by
\begin{equation}\label{e.3.4}
\phi(x_1, ..., x_{n-k}):=\exp(x_1X_{k+1})\cdots\exp(x_{n-k}X_{n})\cdot M.
\end{equation}
Then $\phi$ is an analytic diffeomorphism which is also a measure space isomorphism from $\left( \mathbb{R}^{n-k}, dx \right)$, where $dx$ is the Lebesgue measure, onto $\left( G / M, \mu \right)$, where $\mu$ is the left-invariant measure on $G / M$.
\end{theorem}
Denote

\begin{align}
& \gamma: \mathbb{R}^{n-k} \longrightarrow G, \notag
\\
& \gamma\left( x_{1}, ..., x_{n-k} \right):=\exp\left( x_{1}X_{k+1}\right)\cdot ... \cdot\exp\left( x_{n-k}X_{n}\right), \label{e.3.8}
\end{align}
and recall that the quotient map $q:G \longrightarrow G / M$ in \eqref{e.3.6} is a measure space isomorphism. Then the map $\phi$ in Theorem \ref{t.3.5} can be written as

\[
\phi=q \circ \gamma.
\]
Observe that the map

\begin{align}
& \psi: \mathbb{R}^{n-k} \times M \longrightarrow G, \label{e.3.9}
\\
& \left( x, m\right) \longmapsto \gamma\left( x \right) m \notag
\end{align}
is a diffeomorphism from $\mathbb{R}^{n-k} \times M$ onto $G$. We denote the inverse of this map by

\begin{align}
& \psi^{-1}:=\left( \rho_{1}, \rho_{2} \right), \text{ where } \label{e.3.10}
\\
& \rho_{1}: G \longrightarrow \mathbb{R}^{n-k}, \notag
\\
& \rho_{2}: G \longrightarrow M \notag
\end{align}
Using the diffeomorphism $\psi$ and the fact that by Theorem \ref{t.3.5} the map $\phi$ is a measure space diffeomorphism from $\mathbb{R}^{n-k}$ onto $G$ we can induce a unitary isomorphism

\begin{align}
& J: L^{2}( \mathbb{R}^{n-k}, dx ) \longrightarrow \mathcal{H}_{\pi_{l, \mathfrak{m}}}, \label{e.3.11}
\\
& \left( Jf \right)\left( \gamma\left( x \right), m\right):=\chi_{l}\left( m\right)^{-1}f\left( x\right), \text{ for all } x \in \mathbb{R}^{n-k}, m \in M, f \in L^{2}( \mathbb{R}^{n-k}, dx ). \notag
\end{align}
It is easy to see that $Jf \in \mathcal{H}_{\pi_{l, \mathfrak{m}}}$ since it satisfies \eqref{e.3.2}, and the inverse map is given by

\begin{align*}
& J^{-1}: \mathcal{H}_{\pi_{l, \mathfrak{m}}} \longrightarrow L^{2}( \mathbb{R}^{n-k}, dx ),
\\
& \left( J^{-1}h \right)\left( x\right):=h\left( \gamma\left( x \right)\right), \text{ for all } x \in \mathbb{R}^{n-k}, h \in \mathcal{H}_{\pi_{l, \mathfrak{m}}}.
\end{align*}

Finally, we can describe a unitary representation of $G$ on $L^{2}( \mathbb{R}^{n-k}, dx )$ which is unitarily equivalently to $\pi_{l, \mathfrak{m}}$. Namely,

\begin{align}
&
\left( U_{l, \mathfrak{m}}\left( g \right)f \right)\left( x \right):=\chi_{l}\left( \rho_{2}\left( \gamma\left( x \right)g\right)\right)f\left( \rho_{1}\left( \gamma\left( x \right)g\right)\right), \label{e.3.14}
\\
&
\text{ for all } f\in L^{2}( \mathbb{R}^{n-k}, dx ), g \in G, a.e. x \in \mathbb{R}^{n-k}. \notag
\end{align}
The unitary equivalency of $\pi_{l, \mathfrak{m}}$ and $U_{l, \mathfrak{m}}$ can be shown by using the map $J$
 \[
 J^{-1}\pi_{l, \mathfrak{m}} J=U_{l, \mathfrak{m}}.
 \]
Note that this unitary representation $U_{l, \mathfrak{m}}$ depends on the choice of  the weak Mal'cev basis, but any of these choices gives rise to unitarily equivalent representations.

\begin{remark} We will abuse notation and denote the representation $U_{l, \mathfrak{m}}$ by $\pi_{l, \mathfrak{m}}$ whenever it is clear that the representation space is $L^{2}( \mathbb{R}^{n-k}, dx )$.
\end{remark}

\subsection{Irreducible unitary representations of the group  $G_{n+1}$} Now we apply the induction procedure described in Section \ref{ss.3.4} to the group  $G_{n+1}$ introduced in  Section \ref{ss2.1}.

First we identify a point $(\mathbf{b}, a) \in \mathbb{R}^{n+1}$ with a point in the group $G_{n+1}$ by

 \[
 (\mathbf{b}, a):=\exp(b_nY_n+b_{n-1}Y_{n-1}+\cdots+b_1Y_1+aX).
 \]

\begin{theorem}\label{MM} Let $G_{n+1}$ be the n-step nilpotent Lie group described in Section \ref{ss2.1}, then all (nonequivalent) unitary irreducible representations of $G_{n+1}$ are as follows.
\begin{enumerate}
\item The infinite dimensional representations of $G_{n+1}$ on the Hilbert space $\mathcal{H}_\pi= L^2(\mathbb{R}, dx)$ are given by

\begin{align}
& \pi_{l}(\mathbf{b}, a)f(x):=e^{2\pi i l\left(\sum\limits_{k=1}^nB_k \left( x \right)Y_k\right)}f(x+a), \notag
\\
& B_k \left( x \right):=\sum_{i=1}^k\frac{z_i}{(k-i)!}x^{k-i},   k=1,2,\dots, n, \label{e.3.16}
\end{align}
for $f \in L^2(\mathbb{R}, dx), x \in \mathbb{R}, l \in  \mathcal{O}_{(\alpha,\mathbf{B}_m,0,...,0)},  3 \leqslant m \leqslant n$.
\item The one-dimensional unitary representations (characters) on $\mathcal{H}_\pi=\mathbb C$ are given by
\[
\pi_{l}(\mathbf{b}, a)=e^{2 \pi i (\alpha a+\beta_1 b_1)}I, \ \ \ l\in \mathcal{O}_{( \alpha, \beta_1, 0, ..., 0)}
\]
\end{enumerate}
\end{theorem}
\begin{proof}
First we describe the representations of $G_{n+1}$ corresponding to the two-dimensional orbits, that is,

\begin{align*}
 & \mathcal{O}_{(\alpha, \mathbf{B}_m,0,...,0)}=
 \\
 & \left\{xX^*+\sum_{j=1}^{m-2}f_j(y; \mathbf{B}_m)Y_j^*+yY_{m-1}^*+\beta_mY_m^*\colon x, y\in \mathbb{R}, m \geqslant 2 \right\}
\end{align*}
with $\mathbf{B}_m$ defined by \eqref{e.2.1} and
 \[
 l=\sum_{j=1}^{m-2}f_j(0; \mathbf{B}_m)Y_j^* + \beta_mY_m^*
  \]
being orbit representatives for $\beta_m\not= 0$ for the polarizing subalgebra

\[
 \mathfrak{m}= \operatorname{Span} \left\{Y_1, Y_2,..., Y_n \right\}.
\]
Then for $W=\sum_{i=1}^nb_iY_i$ and this choice of $l$ and $\mathfrak{m}$ we see that as in \eqref{e.3.7}
 \[
 \chi_{l}\left(\exp(W)\right)=e^{2\pi i l(W)}
 \]
defines a character of the subgroup $M=\exp(\mathfrak{m})$. Using \eqref{e.3.9} we identify $G$ with
\[
\mathbb{R} \times \left\{ \exp(x X), x \in \mathbb{R} \right\}
\]
via the diffeomorphism $\psi$. Now we can describe the action of $\pi_{l, \mathfrak{m}}$  on $L^2(\mathbb{R}, dx)$. For $(\mathbf{b}, a)=\exp(b_nY_n + b_{n-1}Y_{n-1} +... + b_1Y_1+aX)$ we have

\begin{align*}
& ( 0, \dots, 0, x)\cdot (\mathbf{b}, a)=
\\
& \left(\sum_{i=1}^nz_i\frac{x^{n-i}}{(n-i)!}, \sum_{i=1}^{n-1}z_i\frac{x^{n-1-i}}{(n-1-i)!}, \dots, z_1, 0\right)\cdot (0, \dots, 0, x+a),
\end{align*}
where $z_{i}$ are defined \ref{e.3.18}. In terms of \eqref{e.3.10}  this identity can be written as

\begin{align*}
& \rho_{1}\left( \gamma\left( x \right) g \right)=\rho_{1}\left( ( 0, \dots, 0, x)\cdot (\mathbf{b}, a) \right)= \rho_{1}\left((0, \dots, 0, x+a)\right)
\\
&=x+a,
\\
&  \rho_{2}\left( \gamma\left( x \right) g \right)=\rho_{2}\left( ( 0, \dots, 0, x)\cdot (\mathbf{b}, a) \right)=
\\
& \rho_{2}\left(\left(\sum_{i=1}^nz_i\frac{x^{n-i}}{(n-i)!}, \sum_{i=1}^{n-1}z_i\frac{x^{n-1-i}}{(n-1-i)!}, \dots, z_1, 0\right)\right) =
\\
& \left(B_{n}\left( x \right), B_{n-1}\left( x \right), ..., B_{1}\left( x \right)\right),
\end{align*}
where
 \[
 B_k\left( x \right)=\sum_{i=1}^kz_i\frac{x^{k-i}}{(k-i)!} ,  \ k=1, 2, \dots, n.
 \]
Identifying $\mathcal{H}_{\pi_{l, \mathfrak{m}}}$ with $L^{2}\left( \mathbb{R}, dx \right)$ we see that
\begin{align*}
\pi_{l, \mathfrak{m}}(\mathbf{b}, a)f(x)&= \chi_{l}\left( \rho_{2}\left( \gamma\left( x \right) g \right)\right)f\left( \rho_{1}\left( \gamma\left( x \right) g \right)\right)
\\
&=e^{2\pi i l(\sum_{k=1}^nB_k\left( x \right)Y_k)}f(x+a).
\end{align*}
Now, for the one-point orbit $\mathcal{O}_{(\alpha,\beta_1,0,\dots,0)}=\{\alpha X^*+\beta_1Y^*_1\}$, $\alpha,\beta_1\in \mathbb{R}$, we have seen that $\mathfrak{m}=\mathfrak{g}_{n+1}$ and induction from $M=\exp(\mathfrak{m})$ is trivial. Thus  $\pi_{l}=\chi_{l}$ is one-dimensional and is given by
\[
\pi_{l}(\mathbf{b},a)=e^{2 \pi i (\alpha a+\beta_1 b_1)}I
\]
with $\mathcal{H}_{\pi_{l}}=\mathbb C$.
\end{proof}

\subsection{Generalized Fourier transform}

We start by reviewing the generalized Fourier transform (GFT) as described in \cite[Section 7.5]{FollandHABook}. First we assume that $G$ is a separable locally compact unimodular Lie group of type I, and $L^p(G,\mathbb{C})$ denotes the space of complex-valued functions on $G$ which are square-integrable with respect to the Haar measure $dg$. Later we consider this setting in the case when $G$ is an addition nilpotent.

Recall that we defined $\widehat{G}$ to be the (unitary) dual space of $G$ in Definition \ref{d.3.5}. The structure of the dual $\widehat{G}$ can be described explicitly for some classes of groups such as locally compact Abelian or compact groups. Another class of groups for which one can find an explicit description of $\widehat{G}$ is of simply connected nilpotent Lie groups as we described in Section \ref{ss.3.4}.

\begin{remark}The GFT is usually defined using the structure of $\widehat{G}$ as a measurable space, and therefore $\widehat{G}$ should have nice properties such as being countably separated, since not countably separated measurable spaces are pathological. By \cite[Theorem 7.6]{FollandHABook} this is equivalent to $G$ being of type I. In particular, if $G$ a simply-connected nilpotent Lie group, then by \cite{Dixmier1959a, Kirillov1962a} the group $G$ is of type I.
\end{remark}

Once we have equipped $\widehat{G}$ with the structure of a nice measurable space, we can define a Borel measure $P$ on $\widehat{G}$ called the \emph{Plancherel measure}. Then a possible issue is how to make a measurable selection of $\pi^{\xi}$ for $\xi \in \widehat{G}$. As observed in \cite[p. 230]{FollandHABook} this can be done if $G$ is of type I. Moreover, if $G$ is a simply-connected nilpotent Lie group, the Plancherel measure can be identified with the Lebesgue measure with a density on $\mathbb{R}^{q}$ for some $q$ using the unitary isomorphism $J$ introduced in \eqref{e.3.11}.

For $\xi \in  \widehat{G}$ we will denote by  $\pi^{\xi}$ a choice of an irreducible representation in the equivalency class $\xi$, and the representation Hilbert space by $\mathcal{H}_{\pi^{\xi}}$ or $\mathcal{H}_{\xi}$.

\begin{definition}\label{d.4.8}
For $f\in L^1(G,\mathbb{C})$ \emph{the generalized Fourier transform (GFT) } of $f$ is the map $\mathcal{F}(f)$ (or $\widehat{f}$) that takes each element of $\widehat{G}$ to a linear operator on the representation space $\mathcal{H}_{\pi^{\xi}}$ by
\begin{equation}\label{e.3.17}
\mathcal{F}(f)\left(\xi\right)=\widehat{f}\left(\xi\right)=\widehat{f}\left(\pi_{\xi}\right):=\int_G f(g)\pi^{\xi}(g^{-1})dg, \text{ for } P-a.e.   \xi \in \widehat{G}.
\end{equation}
\end{definition}
As is known, the Fourier transform  $\mathcal{F}(f)$, $f \in L^1(G,\mathbb{C})\cap L^2(G,\mathbb{C})$, is a Hilbert-Schmidt operator for almost all $\xi \in \widehat{G}$ with respect to the Plancherel measure $P$, and the map $\xi\mapsto \widehat{f}\left(\xi\right)$ is a $P$-measurable field of operators which allows for use of direct integrals, and can be used to find a spectral decomposition of differential operators on $L^2(G,\mathbb{C})$. In particular, one can prove the Plancherel Theorem and the Fourier inversion formula in this abstract setting (e.g. \cite[Theorem 7.44]{FollandHABook}). The Plancherel Theorem gives rise to an extension of the GFT (which we again denote by $\mathcal{F}$) to an isometry

\begin{equation}\label{e.3.12}
\mathcal{F}: L^{2}\left(G,\mathbb{C}\right) \longrightarrow \int_{\widehat{G}} \operatorname{HS}\left( \mathcal{H}_{\pi^{\xi}}\right)dP(\xi),
\end{equation}
where $\operatorname{HS}\left( \mathcal{H}_{\pi^{\xi}}\right)$ is the space of Hilbert-Schmidt operators which is a Hilbert space itself. In the case of $G$ being nilpotent, the Plancherel measure and the space of Hilbert-Schmidt operators can be described explicitly. As it is done over Euclidean spaces, we will use the Fourier transform in Definition \ref{d.4.8}  to find a spectral decomposition for the left-invariant vector field $\widetilde{X}$ corresponding to any $X \in \mathfrak{g}$. In this we will use the following notation. Let $T$ be a linear operator on $L^{2}\left( G, \mathbb{C} \right)$, then we denote by $\widehat{T}$ the following linear operator

\begin{align}
& \widehat{T}:   \int_{\widehat{G}} \operatorname{HS}\left( \mathcal{H}_{\pi^{\xi}}\right)dP(\xi) \longrightarrow \int_{\widehat{G}} \operatorname{HS}\left( \mathcal{H}_{\pi^{\xi}}\right)dP(\xi), \notag
\\
& \widehat{T}=\mathcal{F} T \mathcal{F}^{\ast}, \label{e.3.13}
\end{align}
where $\mathcal{F}^{\ast}$ is the adjoint of $\mathcal{F}$
\[
\mathcal{F}^{\ast}: \int_{\widehat{G}} \operatorname{HS}\left( \mathcal{H}_{\pi^{\xi}}\right)dP(\xi) \longrightarrow L^{2}\left(G,\mathbb{C}\right).
\]
\begin{remark} We would like to mention here that some of the standard properties of the GFT are listed in \cite{AgrachevBoscainGauthierRossi2009}, and a few of them should be interpreted with caution. For example, \cite[Equation (18)]{AgrachevBoscainGauthierRossi2009} gives a formal expression for the GFT of the Dirac mass measure, which needs an introduction of an analogue of tempered distributions. We note here that the theory of tempered distributions have been studied on nilpotent groups (not general unimodular groups), see \cite{David-Guillou2014, Corwin1981a, DhiebLudwig1997}.
\end{remark}

\subsection{The GFT and the Plancherel formula for nilpotent groups}

We start by recalling that for connected, simply connected  nilpotent Lie groups the GFT can be described explicitly. In general for an infinite-dimensional irreducible representation the operator $\pi\left( g \right), g \in G$ is not necessarily a trace-class operator. Let $\mathcal{S}(G)$ be the space of Schwartz functions on $G$ as defined in \cite[Appendix A.2]{CorwinGreenleafBook}, then for $f \in \mathcal{S}(G)$ the operator $\pi\left( f \right)$ defined by \eqref{e.3.17} is a trace-class operator.

By \cite[Theorem 4.2.1]{CorwinGreenleafBook} for the irreducible unitary representation $\pi_{l, \mathfrak{m}}$ on $L^{2}\left( \mathbb{R}^{n-k}, dx \right)$ identified with  $U_{l, \mathfrak{m}}$ by \eqref{e.3.14} we have that for any $f \in \mathcal{S}\left( G \right)$ there is an integral kernel $k_{f} \in \mathcal{S}\left( \mathbb{R}^{n-k} \times \mathbb{R}^{n-k}\right)$ such that for $h \in L^{2}\left( \mathbb{R}^{n-k}, dx \right)$

\[
\left( \widehat{f}\left( \pi_{l, \mathfrak{m}}\right)h \right) \left( x \right)=\int_{\mathbb{R}^{n-k}} h\left( y \right) k_{f}\left( x, y \right) dy.
\]
By \cite[Proposition 4.2.2]{CorwinGreenleafBook} the kernel  $k_{f}\left( x, y \right)$ is given by

\begin{equation}\label{e.3.15}
k_{f}\left( x, y \right)=\int_{M}\chi_{l}\left( m\right) f\left( \gamma\left(x\right)^{-1} m \gamma\left(y\right)\right) dm,
\end{equation}
where $\chi_{l}$ is the character on $M$ defined by \eqref{e.3.7}, $\gamma$ is a map $\mathbb{R}^{n-k} \longrightarrow G$ defined by \eqref{e.3.8}, and $dm$ is the Haar measure on $M$. This is what ter Elst and Robinson call a reduced kernel in \cite[p. 481, (4)]{terElstRobinson1995}.

Now we turn to the Plancherel Formula. Our main goal here is to identify the Plancherel measure with a measure on a Euclidean space.  Let $\{X_1, ... , X_n\}$ be a basis for a nilpotent Lie algebra $\mathfrak{g}$, and let $\{X_1^*, ..., X_n^*\}$ be the dual basis for $\mathfrak{g}^*$. We can use \cite[Theorem 3.1.6, Corollary 3.1.8 ]{CorwinGreenleafBook} to find a $G$-invariant set $U$ of generic coadjoint orbits in $\mathfrak{g}$ and two disjoint sets of indices $S, T$ that partition $\{1,\dots, n\}$ in such a way that (so-called) jump indices of any $l \in U$ is exactly $S$. For more details on jump indices \emph{etc} we refer to \cite[p. 84]{CorwinGreenleafBook}. Let
\begin{align*}
& \mathfrak{g}^*_S:=\operatorname{Span}_{\mathbb{R}}\{ X_i^*, i \in S \},
\\
& \mathfrak{g}^*_T:=\operatorname{Span}_{\mathbb{R}}\{X_j^*, j\in T \}.
\end{align*}
Note that the skew-symmetric form $B_l\left( X, Y \right):=l\left( [ X, Y] \right)$ on $\mathfrak{g}$ has the radical $rad_l$, and therefore induces a non-degenerate skew-symmetric form on $\mathfrak{g} / rad_l$.
\begin{definition}
If $l\in \mathfrak{g}^*$ and $\{X_1, \dots, X_{2k}\}$ is a basis for $\mathfrak{g} / rad_l$, then the \emph{Pfaffian} $\operatorname{Pf}\left( l \right)$ is defined by
\begin{align*}
& \operatorname{Pf}\left( l \right)^2:=\operatorname{det} B^{l}, \text{ where }
\\
& B^{l}(X, Y):=l\left( [X, Y] \right), X, Y \in \mathfrak{g},
\\
& B^{l}_{ij}:=B^{l}(X_i,X_j).
\end{align*}
\end{definition}
First we recall \cite[Theorem 4.3.9]{CorwinGreenleafBook} combined with \cite[p.374, (0.3)]{GoodmanR1971a}, where as before $\mathcal{S}(G)$ is the space of Schwartz functions on $G$.
\begin{theorem}[The Fourier Inversion]\label{t.3.9}
Let $\{X_1,\dots,X_n\}$ be a basis for a nilpotent Lie algebra $\mathfrak{g}$, and let $\{X_1^*,\dots,X_n^*\}$ be the dual basis for $\mathfrak{g}^*$. Define $U, S, T$ and Pfaffian as above, then for $f \in \mathcal{S}(G)$, $f(e)$ is given by an absolutely convergent integral
\[
f\left( g \right)=\int_{U\cap \mathfrak{g}^*_T}\operatorname{Tr}\left( \pi_{l, \mathfrak{m}}\left( g^{-1} \right)\widehat{f}\left(\pi_{l, \mathfrak{m}} \right)\right)|\operatorname{Pf}(l)|dl,
\]
where $dl$ is the Lebesgue measure on $\mathfrak{g}^*_T$.
\end{theorem}
Note that $l \in \mathfrak{g}^*$ can be identified with a point in a Euclidean space by using coordinates of $l$ in the dual basis $\{X_1^*, \dots, X_n^*\}$.  The following theorem can be found in \cite[Theorem 4.3.10]{CorwinGreenleafBook}
\begin{theorem}[Plancherel Theorem]\label{t.3.10} Let notation be as in the previous theorem. For $f \in \mathcal{S}(G)$ we have
\[
\|f\|^2_2=\int_{U\cap \mathfrak{g}^*_T}\|\widehat{f}\left(\pi_{l, \mathfrak{m}} \right)\|^2_{HS}|\operatorname{Pf}(l)|dl,
\]
where $\|A\|_{HS}$ is the Hilbert-Schmidt  norm of an operator $A$.
\end{theorem}

\begin{remark}\label{r.2.16} Note that usually we identify $U\cap \mathfrak{g}^*_T$ with an open subset of a Euclidean space, and therefore $dl$ is simply the Lebesgue measure.
\end{remark}

\begin{definition}\label{d.3.11} The measure $dP:=|\operatorname{Pf}(l)|dl$ is the pushforward of the Plancherel measure on $\widehat{G}$, namely, $P$ is a Borel measure  defined by the identification $\widehat{G}\cong \mathfrak{g}^{\ast}/ \operatorname{Ad}^{\ast} G$. We will abuse notion and call this pushforward measure the \emph{Plancherel measure}.
\end{definition}
In our setting we can explicitly identify the ingredients needed for application of these theorems.

\begin{prop}\label{p.3.13} For $G_{n+1}$ the space $\mathfrak{g}^*_T$ is isomorphic to $\mathbb{R}^{n-1}$ and the Plancherel measure is given by

\[
dP=|\lambda_{n-1}|d\lambda_{1}d\lambda_{2}\cdots d\lambda_{n-1}.
\]
The Plancherel identity then becomes
\[
\|f\|_2^2=\int_{\mathbb{R}^{n-1}}\|\widehat{f}\left(\pi_{l} \right)\|^2_{HS}|\lambda_{n-1}|d\lambda_{1}d\lambda_{2}\cdots d\lambda_{n-1},
\]
where $\pi_{l}$ are the unitary representations described in Theorem \ref{MM} and $l=\left( \lambda_1, ...,  \lambda_{n-1} \right)$ is identified with a point in $\mathbb{R}^{n-1}$.
\end{prop}
\begin{proof}
If $\{Y_n^*,\dots,Y_1^*,X^*\}$ is the dual basis for $\mathfrak{g}_{n+1}^*$, the indices $\{1,\dots,n+1\}$ can be partitioned as $S=\{2,n+1\}$ and $T=\{1,3,4,\dots,n\}$ so that
\[
\left(\mathfrak{g}_{n+1}^*\right)_T=\operatorname{Span}_{\mathbb{R}}\{ Y_n^*, Y_{n-2}^*, \dots, Y_1^* \},
\]
and the generic orbits  are of the form
\[
U=\left\{ \sum_{i=1}^n \alpha_iY_i^*+\alpha X^*\colon \alpha_n\not=0 \right\}.
\]
Thus we see that
\[
\left(\mathfrak{g}_{n+1}^*\right)_T\cap U=\left\{ \lambda_{n-1} Y_n^*+\lambda_{n-2}Y_{n-2}^*+\cdots+\lambda_1Y_1^*\colon   \lambda_i\in \mathbb{R}, \lambda_{n-1}\not= 0\right\}.
\]

The Pfaffian is a polynomial on $\mathfrak{g}_{n+1}^*$ such that
\[
\operatorname{Pf}(l)^2=
\operatorname{det}\left(
  \begin{array}{cc}
    0 & l(Y_n)  \\
   l(-Y_n) & 0 \\
  \end{array}
\right)
=l(Y_n)^2.
\]
Note that $\{Y^*_{n-1}, X^*\}$ is a basis for $\mathfrak{g}_{n+1}/rad_{l}$. Identifying $(\mathfrak{g}_{n+1}^*)_T$ with $\mathbb{R}^{n-1}$ and letting $d\lambda_{1}\cdots d\lambda_{n-1}$ be the Lebesgue measure, we have $|\operatorname{Pf}(l)|=|\lambda_{n-1}|$. If $\pi_{l, \mathfrak{m}}$ is the representation corresponding to
\[
l=\lambda_{n-1} Y_n^*+\lambda_{n-2}Y_{n-2}^*+\cdots+\lambda_1Y_1^*,
\]
the Plancherel formula becomes
\[
\|f\|_2^2=\int_{\mathbb{R}^{n-1}}\|\widehat{f}\left(\pi_{l, \mathfrak{m}} \right)\|^2_{HS}|\lambda_{n-1}|d\lambda_{1}d\lambda_{2}\cdots d\lambda_{n-1},
\]
and the Plancherel measure is identified with the following measure on $\mathbb{R}^{n-1}$

\begin{equation}\label{e.3.5}
dP=|\lambda_{n-1}|d\lambda_{1}d\lambda_{2}\cdots d\lambda_{n-1}.
\end{equation}
\end{proof}

\section{Hypoelliptic heat equation on a sub-Riemannian manifold}\label{s3}

In this section, we start by reviewing some standard definitions in sub-Riemannian geometry, and in particular,  how a natural left-invariant sub-Riemannian structure on nilpotent Lie groups is constructed.

\subsection{Sub-Riemannian manifolds} Let $M$ be an $n$-dimensional connected smooth manifold, with tangent and cotangent bundles $TM$ and $T^*M$ respectively.
\begin{definition}
For   $m \leqslant n$, let $\mathcal{H}$ be a smooth sub-bundle of $TM$, where each fiber $\mathcal{H}_q$ has dimension $m$ and is equipped with an inner product $\mathbf{g}$ which smoothly varies between fibers. Then

\begin{enumerate}
\item the triple $\left( M, \mathcal{H}, \mathbf{g} \right)$ is called a \emph{sub-Riemannian manifold of rank} $m$;

\item $\mathcal{H}$ is called a \emph{horizontal distribution} on $M$, and  the inner product $\mathbf{g}$ a \emph{sub-Riemannian metric};

\item sections of $\mathcal{H}$ are called \emph{horizontal vector fields}, and curves on $M$ whose velocity vectors are always horizontal are called \emph{horizontal curves}.
\end{enumerate}
\end{definition}

\begin{as}[H\"{o}rmander's condition]\label{a.3.2}
Throughout this paper we assume that the distribution $\mathcal{H}$ satisfies H\"{o}rmander's (bracket generating) condition; that is, horizontal vector fields with their Lie brackets span the tangent space $T_{p}M$ at every point $p \in M$.
\end{as}
Under H\"{o}rmander's condition any two points on $M$ can be connected by a horizontal curve by the Chow-Rachevski theorem. Thus there is a natural \emph{sub-Riemannian distance} (\emph{Carnot-Carath\'{e}odory distance}) on $M$ defined as the infimum over the lengths of horizontal curves connecting two points. In turn, this affords us the notion of a \emph{horizontal geodesic}, a horizontal curve whose length (locally) realizes the Carnot-Carath\'{e}odory distance.

\begin{definition} Let $\left( M, \mathcal{H}, \mathbf{g} \right)$ be a sub-Riemannian manifold.

\begin{enumerate}

\item Locally $(\mathcal{H}, \mathbf{g})$ can be given by assigning a set of $k$ smooth vector fields spanning $\mathcal{H}$. Suppose these vector fields can be chosen so that they are orthonormal with respect to $\mathbf{g}$
\[
\mathcal{H}_{p}=\operatorname{Span}\{X_1(p), \dots, X_{m}(p) \},  \mathbf{g}_p\left( X_i(p), X_j(p) \right)=\delta_{ij},
\]
then we say that the sub-Riemannian manifold $M$ is \emph{trivializable}.

\item If $M$ is analytic and the vectors $\{X_1, \dots, X_{m}\}$ are analytic vector fields, we say the sub-Riemannian manifold $\left( M, \mathcal{H}, \mathbf{g} \right)$ is \emph{analytic}.

\item $M$ is \emph{regular} if for

\[
\mathcal{H}_1:= \mathcal{H},  \mathcal{H}_{i+1} := \mathcal{H}_i+[\mathcal{H}_i,\mathcal{H}]
 \]
the dimension of $\mathcal{H}_i(p)$, $i=1, 2, \dots$ does not depend on the point $p\in M$.
\end{enumerate}

\end{definition}

\subsection{Left-invariant sub-Riemannian structure on Lie groups}

Let $G$ be a Lie group with Lie algebra $\mathfrak{g}:=T_{e}G$. As usual we identify $T_{e}G$ with the set of  left-invariant vector fields on $G$ as follows. First consider smooth diffeomorphisms $L_{a}$ and $R_{a}$ of $G$, namely, the left and right translations by an element $p \in G$

\begin{align*}
& L_{p}: h \mapsto ph, h \in G,
\\
& R_{p}: h \mapsto hp, h \in G.
\end{align*}

\begin{notation}\label{n.4.5} For any $X \in \mathfrak{g}$  we denote by  $\widetilde{X}$ the unique left-invariant vector field such that $\widetilde{X}\left( e \right)=X$, that is, for any $f \in C^{\infty}\left( G \right)$

\begin{equation}\label{e.4.20}
\widetilde{X}f\left( g \right):=\left.\frac{d}{dt}\right|_{t=0}f\left( ge^{tX}\right)=\left.\frac{d}{dt}\right|_{t=0}f\left( R_{e^{tX}}g\right),
\end{equation}
where $R_{p}: G \to G$ is the right translation by an element $p \in G$
\end{notation}

The Lie group $G$ can  be equipped with a left-invariant Riemannian metric as follows. First an inner product $\langle \cdot, \cdot \rangle_{e}$ on $T_{e}G=\mathfrak{g}$ determines  a left-invariant  metric on $G$ by  the identification $dL_{p} : \mathfrak{g}\cong T_{e}G \rightarrow T_{p}G$, for any $p \in G$. Then we can define the corresponding Riemannian metric $\langle \cdot, \cdot \rangle$ on $G$ by

\begin{equation}\label{e.4.1}
\langle X, Y \rangle_{p}:=\langle \left( dL_{p^{-1}}\right)_{p}\left( X \right), \left( dL_{p^{-1}}\right)_{p}\left( Y \right) \rangle_{e}
\end{equation}
for any $p \in G$ and $X, Y \in \mathfrak{g}$. Another way to write this is as the pull-back $\left( dL_{p}\right)^{\ast}\mathbf{g}=\mathbf{g}$.
Conversely, if a Riemannian metric on $G$  satisfies \eqref{e.4.1}, then it is called left-invariant.

Now we can describe a similar construction in the sub-Riemannian setting.
\begin{definition}
Let $G$ be a Lie group with Lie algebra $\mathfrak{g}$, $\mathcal{H}$ be a distribution satisfying H\"{o}rmander's condition (Assumption \ref{a.3.2}), and $\mathbf{g}$ be a sub-Riemannian metric. We will say that $\left( G, \mathcal{H}, \mathbf{g} \right)$ is equipped with a \emph{left-invariant sub-Riemannian structure} if

\begin{enumerate}
\item
the distribution $\mathcal{H}$ is left-invariant, that is, $\mathcal{H}_{e}$ is a linear subspace of the Lie algebra  $\mathfrak{g}$  such that
\[
\mathcal{H}_{p} = L_{p}\mathcal{H}_{e} \text{ for } p \in G;
\]

\item the metric $\mathbf{g}$ is left-invariant, that is,
\[
\mathbf{g}_{p}\left( X, Y \right)=\mathbf{g}_{e}\left( \left( dL_{p^{-1}}\right)_{p}\left( X \right), \left( dL_{p^{-1}}\right)_{p}\left( Y \right) \right), p \in G, X, Y \in \mathfrak{g}.
\]

\end{enumerate}
\end{definition}

Finally, we observe that all left-invariant sub-Riemannian manifolds are regular and trivializable according to \cite[Remark 5, p. 2627]{AgrachevBoscainGauthierRossi2009}.

\subsection{The hypoelliptic Laplacian}

As we observed in the introduction, there are several choices that need to be made to define a natural hypoelliptic Laplacian corresponding to the sub-Riemannian structure at hand. One of these choices is of a measure, and as is seen from \cite[Remark 16]{AgrachevBoscainGauthierRossi2009} in the case of Lie groups equipped with a left-invariant sub-Riemannian  structure  the Haar measure is a natural choice: the Popp measure defined in \cite{MontgomeryBook2002} and the Hausdorff measure are both left-invariant, and therefore are proportional to the left Haar measure.

Once the measure is chosen, an operator defined as a divergence of the horizontal gradient is the usual sum of squares operator. Note that the argument in \cite{AgrachevBoscainGauthierRossi2009} has a mistake, and for more detailed discussion of related issues we refer to \cite{GordinaLaetsch2014a, GordinaLaetsch2014b}. But no matter which point of view we use, in the case of a nilpotent Lie group all these approaches give the same result: the sum of squares operator.

Let $(G,\mathcal{H}, \mathbf{g})$  be a unimodular Lie group  equipped with a left-invariant sub-Riemannian structure of rank $k$. Our goal now is to see how we can use the GFT to diagonalize left-invariant vector fields, and therefore the sub-Laplacian $\Delta_{H}$. This is really a non-commutative analogue of the Euclidean case. One of the issues we want to clarify in what follows is the domains of the operators involved. This is something that is missing in \cite{AgrachevBoscainGauthierRossi2009}, and can be made explicit in the case when $G$ is nilpotent.

Define $\Delta_H$ to be the left-invariant second order differential operator
\[
\Delta_H f:=\sum_{i=1}^{m}\widetilde{X_{i}}^{2} f, \ f \in C_{c}^{\infty}\left( G \right),
\]
where $\{ X_{i}\}_{i=1}^{m}$ is an orthonormal basis of $\left( \mathcal{H}_{e}, \mathbf{g}_{e} \right)$.

The operator $\Delta_H$ is a densely defined symmetric operator on $L^{2}\left( G, dg \right)$, where $dg$ is a right-invariant Haar measure. It has a self-adjoint extension, namely, the Friedrichs extension, which we will denote by the same $\Delta_H$. For details we refer to \cite[Section II.5]{VaropoulosBook1992}. In addition we assume that  $\mathcal{H}$ satisfies H\"{o}rmander's condition \ref{a.3.2}, therefore the operator $\Delta_H$ is hypoelliptic by \cite{Hormander1967a}.

\begin{definition}\label{d.4.6} Let $P_{t}$ denote the heat semigroup $e^{t\Delta_H}$, where $\Delta_H$ is the self-adjoint (Friedrichs) extension of $\Delta_H|_{C_{c}^{\infty}\left( G \right)}$ to $L^{2}\left( G, dg \right)$, where $dg$ is a right-invariant Haar measure. By the left invariance of $\Delta_H$ and H\"{o}rmander condition \eqref{a.3.2}, $P_{t}$ admits a left convolution kernel $p_{t}$  such that

\[
P_{t}f\left( h\right) = f \ast p_{t}\left( h\right) = \int_{G} f\left( hg\right) p_{t}\left( g\right) dg
\]
for all $ f \in C_{c}^{\infty}\left( G \right)$. The function $p_{t}$ is called the \emph{hypoelliptic heat kernel} of $G$.
\end{definition}

\subsection{Generalized Fourier transform}

We will use the Fourier transform in Definition \ref{d.4.8} and the Fourier inversion formula to find a spectral decomposition for the left-invariant vector field $\widetilde{X}$ corresponding to any $X \in \mathfrak{g}$.

To describe a domain for the differential operator $\widetilde{X}$, as well as for the Fourier transform of  $\widetilde{X}$, we turn to the notion of smooth and analytic vectors for a representation (due to E.~Nelson \cite{Nelson1959a}, see also \cite{GoodmanR1969a}).

\begin{definition}\label{d.4.7}
Let $\pi$ be a unitary representation of $G$ in a complex Hilbert space $\mathcal{H}_{\pi}$, consider the map $f_{v}: G \to \mathcal{H}_{\pi}$ given by

\begin{equation}\label{e.4.3}
f_{v}\left( g \right):=\pi\left( g \right)v \text{ for } v \in \mathcal{H}_{\pi}.
\end{equation}
Then

\begin{enumerate}
  \item $v$ is an \emph{analytic vector} if the map $f_{v}$ is analytic, and the space of such vectors is denoted by $\mathcal{H}_{\pi}^{\omega}$;
  \item $v$ is $C^{k}$ \emph{vector}, if the map $f_{v}$ is $C^{k}$, and the space of such vectors is denoted by $\mathcal{H}_{\pi}^{k}$;
  \item $\mathcal{H}_{\pi}^{\infty}:=\bigcap_{k=0}^{\infty}\mathcal{H}_{\pi}^{k}$, then $v \in \mathcal{H}_{\pi}^{\infty}$ is called a ($C^{\infty}$) \emph{smooth vector}.
\end{enumerate}

\end{definition}

It is well-known that $\mathcal{H}_{\pi}^{\omega}$ and  $\mathcal{H}_{\pi}^{\infty}$ are dense linear subspaces of $\mathcal{H}_{\pi}$, and both are $\pi\left( G\right)$-invariant (\cite[p. 230]{CorwinGreenleafBook}). By \cite[Theorem 3.3]{DixmierMalliavin1978} the space $\mathcal{H}_{\pi}^{\infty}$ coincides with the G{\aa}rding space of finite sums of vectors $\widehat{f}\left( \pi \right)v$, $f \in C_{c}^{\infty}\left( G \right)$, $v \in \mathcal{H}_{\pi}$. We will use the space $\mathcal{H}_{\pi}^{\infty}$ because it works better with the GFT. In particular, we have that for $\pi \in \widehat{G}$

\begin{equation}\label{e.4.4}
\widehat{f}\left( \pi \right)v \in \mathcal{H}_{\pi}^{\infty}, \text{ for all } v\in \mathcal{H}_{\pi}^{\infty}, f \in \mathcal{S}\left( G \right),
\end{equation}
and  if $G$ is a connected, simply connected nilpotent group, then

\begin{equation}\label{e.4.5}
 \mathcal{H}_{J^{-1}\pi J}^{\infty}=\mathcal{S}\left( \mathbb{R}^{n-k} \right),
 \end{equation}
where the first statement can be found in \cite[Theorem A.2.7]{CorwinGreenleafBook}, the second in \cite[Corollary 4.1.2]{CorwinGreenleafBook} with $J$ being the unitary isomorphism introduced in \eqref{e.3.11}.

\begin{definition}\label{d.4.9} Let $X \in \mathfrak{g}$, $\pi \in \widehat{G}$, then define the differential operator $d\pi\left( X \right)$  on $\mathcal{H}_{\pi}$ with the domain

\[
\mathcal{D}\left( d\pi\left( X \right)\right):=\left\{ v \in \mathcal{H}_{\pi}: d\pi(X)\left( v \right):=\widetilde{X}f_{v}=\left.\frac{d}{dt}\right|_{t=0}f_{v}\left( e^{tX} \right) \text{ exists } \right\},
\]
where $f_{v}\left( e^{tX} \right)=\pi\left( e^{tX} \right)v$ as defined by \eqref{e.4.3}.
\end{definition}
Note that since $\pi$ is a unitary representation, by Stone's theorem $d\pi\left( X \right)$ is a closed,  densely defined, essentially skew-adjoint operator on $\mathcal{H}_{\pi}$. Moreover, as noted in \cite[p.226]{CorwinGreenleafBook} for any $X \in \mathfrak{g}$ we have

\begin{align}
& \mathcal{H}_{\pi}^{\infty} \subseteq \mathcal{D}\left( d\pi\left( X \right)\right), \notag
\\
& d\pi(X)\left( \mathcal{H}_{\pi}^{\infty}\right)\subseteq \mathcal{D}\left( d\pi\left( X \right)\right),\label{e.4.6}
\\
& \widetilde{X}f_{v}\left( g \right)=f_{d\pi(X)\left( v \right)}\left(  g \right), v \in \mathcal{H}_{\pi}^{\infty}, g \in G. \notag
\end{align}

\begin{remark}Now we are ready to comment on \cite[Proposition 24]{AgrachevBoscainGauthierRossi2009}. Note that to make sense of all the ingredients in the statement, we need to know the domains of differential operators involved. If we start with $\mathcal{H}_{\pi}^{\omega}$ as the domain of $d\pi(X)$, we see that for a general unimodular group $d\pi\left( X \right)\circ d\pi\left( X \right)$ is again a densely defined second-order differential operator on $\mathcal{H}_{\pi}^{\omega}$. Observe that this domain might depend of the representation $\pi$ which makes interpreting \cite{AgrachevBoscainGauthierRossi2009} difficult.

There are some groups $G$ for which this construction can be made rigorous. The first example is of the Heisenberg group $H_{3}$. We refer to \cite[p. 65]{GoodmanR1969a} for details, but the main point is that infinite-dimensional irreducible unitary representations $\pi_{\lambda} \in \widehat{H_{3}}$ can be indexed by $\lambda \in \mathbb{R}$, and the Plancherel measure can be described as a measure on $\mathbb{R}$. In this case $\mathcal{H}_{\pi_{\lambda}}^{\omega}=\mathcal{H}_{\pi_{1}}^{\omega}$, and so is independent of $\lambda$. We describe this group in Example \ref{ex.4.1}.
This construction can be extended to other nilpotent groups as well, and in particular we can use the fact that $d\pi\left( X \right)$ are differential operators with common domains if $\pi \in \widehat{G}$ are realized as unitary representations on $L^{2}\left( \mathbb{R}^{n-k}, dx \right)$.
\end{remark}
The next theorem can be viewed as a rigorous version of \cite[Theorem 26]{AgrachevBoscainGauthierRossi2009}. In this case $G$ is not assumed to be nilpotent.

\begin{theorem}\label{t.4.10} For any $X \in \mathfrak{g}$, $\pi \in \widehat{G}$, $v \in \mathcal{H}_{\pi}^{\infty}$, $f \in \mathcal{S}\left( G \right)$, then

\[
\widehat{\widetilde{X}^{2}f}\left( \pi \right)v=d^{2}\pi\left( X \right)\widehat{f}\left(\pi \right)v.
\]
Thus $\widehat{\Delta_{H} f}\left( \pi \right)$ is a close, densely defined self-adjoint operator on $\mathcal{H}_{\pi}$ with the domain $\mathcal{H}_{\pi}^{\infty}$ such that for $f \in \mathcal{S}\left( G \right)$

\begin{equation}
\widehat{\Delta_{H} f}\left( \pi \right)=\left( \sum_{i=1}^{m}d^{2}\pi\left( X_{i} \right)\right)\widehat{f}\left(\pi \right).
\end{equation}

\end{theorem}

\begin{proof} The proof is based on properties of the GFT that we formulated earlier. First, we observe that for operators $\widehat{f}\left( \pi \right)$ and $d\pi\left( X \right)$ by \eqref{e.4.4} and \eqref{e.4.6} for $f \in \mathcal{S}\left( G \right)$ we have

\[
\widehat{f}\left( \pi \right)\left( \mathcal{H}_{\pi}^{\infty}\right)\subseteq \mathcal{H}_{\pi}^{\infty}\subseteq \mathcal{D}\left( d\pi\left( X \right)\right),
\]
Finally, similarly to \cite[p. 124]{CorwinGreenleafBook} for $f \in C_{c}^{\infty}\left( G \right)$

\begin{align*}
& \widehat{\widetilde{X}f}\left( \pi \right)=\left.\frac{d}{dt}\right|_{t=0} \int_{G} f\left( ge^{tX} \right)\pi\left( g^{-1}\right)dg=
\\
& \left.\frac{d}{dt}\right|_{t=0} \int_{G} f\left( g \right)\pi\left( e^{-tX} g^{-1}\right)dg=
\\
& \left.\frac{d}{dt}\right|_{t=0} \int_{G} f\left( g \right)\pi\left( e^{-tX} \right)\pi\left(  g^{-1}\right)dg=-d\pi\left( X \right)\widehat{f}\left( \pi \right),
\end{align*}
where this limit exists on $\mathcal{D}\left( d\pi\left( X \right)\right)$, and so in particular on $\mathcal{H}_{\pi}^{\infty}$ by \eqref{e.4.6}.

The result now follows since we can apply the same argument on the space of smooth vectors $\mathcal{H}_{\pi}^{\infty}$.
\end{proof}
Using \eqref{e.3.13} we can say that
\[
\widehat{\Delta_{H}}: \int_{\widehat{G}} \operatorname{HS}\left( \mathcal{H}_{\pi^{\xi}}\right)dP(\xi) \longrightarrow \int_{\widehat{G}} \operatorname{HS}\left( \mathcal{H}_{\pi^{\xi}}\right)dP(\xi)
 \]
is an (essentially) self-adjoint operator for each $\pi^{\xi}$ which acts by multiplication by the operator $\left( \sum_{i=1}^{m}d^{2}\pi^{\xi}\left( X_{i} \right)\right)$ on the space of $\operatorname{HS}\left( \mathcal{H}_{\pi^{\xi}}\right)$. If in addition $G$ is a connected, simply connected nilpotent group, then for $\pi=\pi_{l, \mathfrak{m}}$ and $f \in C_{c}^{\infty}\left( G \right)$

\[
\widehat{\Delta_{H} f}\left( \pi_{l, \mathfrak{m}} \right)=\left( \sum_{i=1}^{m}d^{2}\pi_{l, \mathfrak{m}}\left( X_{i} \right)\right)\widehat{f}\left(\pi_{l, \mathfrak{m}} \right)
\]
can be described more explicitly. By \eqref{e.4.5} the operator $\sum\limits_{i=1}^{m}d^{2}\pi_{l, \mathfrak{m}}\left( X_{i} \right)$ can be identified with an operator on $L^{2}\left( \mathbb{R}^{n-k}, dx \right)$ which we denote by

\begin{equation}\label{e.4.9}
\widehat{\Delta_{H}}\left( \pi_{l, \mathfrak{m}} \right):=\sum\limits_{i=1}^{m}d^{2}\pi_{l, \mathfrak{m}}\left( X_{i} \right).
\end{equation}

Now we can see that in the nilpotent case there is a natural semigroup corresponding to the GFT of the hypoelliptic Laplacian $\widehat{\Delta_{H}}\left( \pi_{l, \mathfrak{m}} \right)$, which in turn gives an explicit formula for the hypoelliptic heat kernel. Note that \cite[Corollary 29]{AgrachevBoscainGauthierRossi2009} can be interpreted as a version of this formula (modulo the issues we mentioned previously), and in the nilpotent case we also refer to \cite[Equation (6), p. 484]{terElstRobinson1995}. 
\begin{theorem}[Hypoelliptic heat kernel]\label{t.4.11}

\begin{equation}\label{e.3.20}
p_{t}\left( g \right)=\int_{U\cap \mathfrak{g}^*_T} \int_{ \mathbb{R}^{n-k}}\chi_{l}\left( \rho_{2}\left( \gamma\left( x \right)g\right)\right)k_{t}^{l}\left( \rho_{1}\left( \gamma\left( x \right)g\right), x\right)dx \vert\operatorname{Pf}(l)\vert dl,
\end{equation}
where $k_{t}^{l}\left( x, y \right)$ is the heat kernel for a continuous semigroup with the generator $\widehat{\Delta_{H}}\left( \pi_{l, \mathfrak{m}} \right)$. This operator is a second order differential operator with polynomial coefficients on $L^{2}\left( \mathbb{R}^{n-k}, dx \right)$ and with a nonnegative polynomial potential whose degree depends on the structure of the group $G$.
\end{theorem}

\begin{remark} By Remark \ref{r.2.16} we usually identify $U\cap \mathfrak{g}^*_T$ with an open subset of a Euclidean space and $dl$ with the Lebesgue measure. Thus \eqref{e.3.20} gives a Euclidean integral formula for the hypoelliptic heat kernel on $G$. In addition, we can view $U\cap \mathfrak{g}^*_T$ as a non-commutative spectrum of a nilpotent group $G$. This goes back to an observation in \cite{Kirillov1962a} that no discrete spectra arises in this case, and so it is not surprising that we have an integral formula instead of a series as for a compact Lie group (e.~g. \cite{BaudoinBonnefont2009}).
\end{remark}

\begin{proof} Recall that the heat kernel $p_{t}$ is the convolution kernel defined in Definition \ref{d.4.6}
\[
P_{t}f\left( g\right) = f \ast p_{t}\left( g\right) = \int_{G} f\left( gh\right) p_{t}\left( h\right) dh.
\]
The heat kernel $p_{t} \in \mathcal{S}\left( G \right)$, and therefore by \eqref{e.3.15} for the representation  $\pi_{l, \mathfrak{m}}$ on $L^{2}\left( \mathbb{R}^{n-k}, dx \right)$ there is an integral kernel $k_{t} \in \mathcal{S}\left( \mathbb{R}^{n-k} \times \mathbb{R}^{n-k}\right)$ such that for $h \in L^{2}\left( \mathbb{R}^{n-k}, dx \right)$

\begin{equation}\label{e.4.7}
\left( \widehat{p_{t}}\left( \pi_{l, \mathfrak{m}}\right)h \right) \left( x \right)=\int_{\mathbb{R}^{n-k}} h\left( y \right) k_{t}^{l}\left( x, y \right) dy,
\end{equation}
where

\begin{equation}\label{e.4.8}
k_{t}^{l}\left( x, y \right)=\int_{M}\chi_{l}\left( m\right) p_{t}\left( \gamma\left(x\right)^{-1} m \gamma\left(y\right)\right) dm, \hskip0.07in x, y \in \mathbb{R}^{n-k}.
\end{equation}
Here we use Vergne's polarization subalgebra $\mathfrak{m}$ to define $M$.  First observe that for any $f \in \mathcal{S}\left( G \right)$

\[
\widehat{P_{t}f}\left( \pi_{l, \mathfrak{m}}\right)=\widehat{f}\left( \pi_{l, \mathfrak{m}}\right)\circ \widehat{p_{t}}\left( \pi_{l, \mathfrak{m}}\right)
\]
is a trace-class integral operator on $L^{2}\left( \mathbb{R}^{n-k}, dx \right)$. To recover the heat kernel, we need to find a function $f$ such that $\widehat{f}\left( \pi_{l, \mathfrak{m}}\right)$ is close to the identity operator on $\mathcal{H}_{l, \mathfrak{m}}$. This presents two problems: the first is that the identity operator is not a Hilbert-Schmidt operator, and the second is that in general $\widehat{p_{t}}\left( \pi_{l, \mathfrak{m}}\right)$ is only Hilbert-Schmidt, and therefore taking the trace in the Fourier inversion formula might be problematic. We will deal with the first issue by taking an approximate identity in $G$ which is equivalent to defining the Dirac $\delta$ function as a tempered distribution, and the second issue has been addressed in \cite{CorwinGreenleafBook} as we explain below.

Namely, let $\{ \varphi_{n}\}_{n=1}^{\infty}$ be a bounded approximate identity, that is, $\varphi_{n} \in C_{c}\left( G \right)$  be a sequence of functions such as in \cite[Proposition 2.42]{FollandHABook}. In particular, for any $h \in L^{2}\left( G \right)$

\[
h \ast \varphi_{n} \xrightarrow[n \to \infty]{L^{2}\left( G \right)} h
\]
and the operator norms of $h \longmapsto h \ast \varphi_{n}$ are (uniformly) bounded.

Now we can use the fact that by \cite[Theorem 4.2.1]{CorwinGreenleafBook} for any function in $\mathcal{S}\left( G \right)$ its Fourier transform is not just Hilbert-Schmidt, but trace-class. Therefore for the heat kernel $p_{t}$ we see that $ \widehat{p_{t}}\left( \pi_{l, \mathfrak{m}}\right)$ is a trace-class operator, and because of the assumptions on the approximate identity $\{ \varphi_{n}\}_{n=1}^{\infty}$  we see that the convergence under the trace

\[
\operatorname{Tr}\left( \widehat{\varphi_{n}}\left( \pi_{l, \mathfrak{m}}\right)\circ \widehat{p_{t}}\left( \pi_{l, \mathfrak{m}}\right)\right)
\xrightarrow[n \to \infty]{}
\operatorname{Tr}\left( \widehat{p_{t}}\left( \pi_{l, \mathfrak{m}}\right)\right)
\]
holds. Now we can use the Fourier Inversion formula \cite[Theorem 7.44, p. 234]{FollandHABook} for nice enough functions $h$ (such as $p_{t} \ast \varphi_{n}$) on $G$

\[
h\left( g\right)= \int_{\widehat{G}}\operatorname{Tr}(\widehat{h}(\pi^{\xi})\circ \pi^{\xi}(g))dP(\xi).
\]
In particular, if $h$ is in addition continuous, this is a pointwise identity.  Applying this formula to $P_{t}\varphi_{n}$, we have
\begin{align*}
& \left( P_{t}\varphi_{n}\right)\left( g\right)= \left( p_{t} \ast \varphi_{n}\right)\left( g\right)=
\\
& \int_{\widehat{G}}\operatorname{Tr}(\widehat{p_{t} \ast \varphi_{n}}(\pi^{\xi})\circ \pi^{\xi}(g))dP(\xi)=
\\
& \int_{\widehat{G}}\operatorname{Tr}\left( \widehat{\varphi_{n}}\left( \pi^{\xi}\right)\circ \widehat{p_{t}}\left( \pi^{\xi}\right)\circ \pi^{\xi}(g) \right)dP(\xi).
\end{align*}
In our setting the trace is taken in the representation space $L^{2}\left( \mathbb{R}^{m}, dx \right)$, and by Theorem \ref{t.3.10} we can identify $\widehat{G}$ with $U\cap \mathfrak{g}^*_T$ equipped with the Plancherel measure $\vert\operatorname{Pf}(l)\vert dl$, so we have the Fourier Inversion formula for $P_{t}\varphi_{n}$ gives

\begin{align*}
& \left( P_{t}\varphi_{n}\right)\left( g\right)=\int_{U\cap \mathfrak{g}^*_T}\operatorname{Tr}(\widehat{P_{t}\varphi_{n}}(\pi_{l, \mathfrak{m}})\circ \pi_{l, \mathfrak{m}}(g))\vert\operatorname{Pf}(l)\vert dl=
\\
& \int_{U\cap \mathfrak{g}^*_T}\operatorname{Tr}(\widehat{\varphi_{n}}\left( \pi_{l, \mathfrak{m}}\right)\circ \widehat{p_{t}}\left( \pi_{l, \mathfrak{m}}\right)\circ \pi_{l, \mathfrak{m}}(g))\vert\operatorname{Pf}(l)\vert dl \xrightarrow[n \to \infty]{}
\\
& \int_{U\cap \mathfrak{g}^*_T}\operatorname{Tr}( \widehat{p_{t}}\left( \pi_{l, \mathfrak{m}}\right)\circ \pi_{l, \mathfrak{m}}(g))\vert\operatorname{Pf}(l)\vert dl=
\\
& \int_{U\cap \mathfrak{g}^*_T}\operatorname{Tr}( \pi_{l, \mathfrak{m}}(g)\circ\widehat{p_{t}}\left( \pi_{l, \mathfrak{m}}\right))\vert\operatorname{Pf}(l)\vert dl,
\end{align*}
where we used the fact that the convergence under $\operatorname{Tr}$ and therefore under the integral does not change after composing the unitary operator $\pi_{l, \mathfrak{m}}(g)$, and the last line uses the centrality of the operator trace. Now observe that if we have a unitary operator $U$ on $L^{2}\left( \mathbb{R}^{n-k} \right)$ composed with a trace class integral operator

\[
\left( Kf\right)\left( x \right):= \int_{\mathbb{R}^{n-k}}f \left( y \right) k\left( x, y \right) dy,
\]
then $U\circ K$ is a trace class integral operator with the kernel given by

\[
U k\left( x, \cdot \right),
\]
where $U$ is applied to the kernel $k\left( x, y \right)$ in the variable $x$. Then

\[
\operatorname{Tr} \left( U\circ K \right)=\int_{\mathbb{R}^{n-k}}U k\left( x, x \right) dx.
\]
Now we can refer to \eqref{e.4.7} and \eqref{e.4.8} to see that

\[
\left( \widehat{p_{t}}\left( \pi_{l, \mathfrak{m}}\right)h \right) \left( x \right)=\int_{\mathbb{R}^{n-k}} h\left( y \right) k_{t}^{l}\left( x, y \right) dy,
\]
where $k_{t}$ is defined by \eqref{e.4.8}.

Recall that we identify the unitary representation $\pi_{l, \mathfrak{m}}$ with the representation $U_{l, \mathfrak{m}}$ defined by \eqref{e.3.14}, and so applying it to the kernel $k_{t}^{l}$ in $x$ we have

\[
\left( \pi_{l, \mathfrak{m}}\left( g \right)k_{t}^{l} \right)\left( x, y \right)=\chi_{l}\left( \rho_{2}\left( \gamma\left( x \right)g\right)\right)k_{t}^{l}\left( \rho_{1}\left( \gamma\left( x \right)g\right), y\right),
\]
and therefore
\begin{align*}
&  \left( P_{t}\varphi_{n}\right)\left( g\right)\xrightarrow[n \to \infty]{}
\\
& \int_{U\cap \mathfrak{g}^*_T} \int_{ \mathbb{R}^{n-k}}\chi_{l}\left( \rho_{2}\left( \gamma\left( x \right)g\right)\right)k_{t}^{l}\left( \rho_{1}\left( \gamma\left( x \right)g\right), x\right)dx \vert\operatorname{Pf}(l)\vert dl.
\end{align*}
Now we recall that $p_{t}$ is a convolution kernel, and thus at least

\[
\left( P_{t}\varphi_{n}\right)\left( g\right)\xrightarrow[n \to \infty]{L^{2}\left( G \right)} p_{t}\left( g \right).
\]
Actually, this convergence is pointwise since the hypoelliptic heat kernel $p_{t}$ is continuous and bounded.
Thus
\[
p_{t}\left( g \right)=\int_{U\cap \mathfrak{g}^*_T} \int_{ \mathbb{R}^{n-k}}\chi_{l}\left( \rho_{2}\left( \gamma\left( x \right)g\right)\right)k_{t}^{l}\left( \rho_{1}\left( \gamma\left( x \right)g\right), x\right)dx \vert\operatorname{Pf}(l)\vert dl.
\]
Finally we observe that $k_{t}^{l}$ is the heat kernel for $\widehat{\Delta_{H}}\left( \pi_{l, \mathfrak{m}} \right)$ on $L^{2}\left( \mathbb{R}^{n-k}, dx \right)$ by \cite[p. 481]{terElstRobinson1995}. Note that this heat kernel corresponds to the initial point $\gamma^{-1}\left( x \right)\gamma\left( y \right)$ as can be seen from \eqref{e.4.8}. Recall that we can use Mal'cev basis which allows us to write left-invariant vector fields as first order differential operators with polynomial coefficients. By \cite[Theorem 4.1.1]{CorwinGreenleafBook} in these coordinates $d\pi_{l, \mathfrak{m}}\left( X_i\right)$ is a differential operator of degree $0$ or $1$ with polynomial coefficients with respect to $\left( x_1, x_2,..., x_n \right)$, and therefore the operator  $\widehat{\Delta_{H}}\left( \pi_{l, \mathfrak{m}} \right)$ is a second order differential operator with polynomial coefficients on $L^{2}\left( \mathbb{R}^{n-k}, dx \right)$.
\end{proof}

\begin{remark}\label{r.4.12} The kernel $k_{t}^{l}\left( x, y \right)$ is called  the reduced heat kernel in \cite[p. 481]{terElstRobinson1995} among other papers, and it is the kernel of the following semigroup

\[
\left( e^{t\widehat{\Delta_{H}}\left( \pi_{l, \mathfrak{m}} \right)}f\right)\left( x \right)=\int_{ \mathbb{R}^{n-k}}k_{t}^{l}\left( x, y \right) f\left( y \right) dy, f \in L^{2}\left( \mathbb{R}^{n-k}, dy \right).
\]
\end{remark}

\begin{remark} Note that we can use some crude estimates for the heat kernel for the Schr\"{o}dinger operator such as in \cite{DaviesHeat_Kernels_and_Spectral_Theory} to prove heat kernel estimates for the hypoelliptic heat kernel $p_{t}$.  We use the following statement in \cite[Proposition 2.2.8]{SpinaChiaraPhDThesis}. For example, the reduced heat kernel $k_{t}^{l}$ in the case when $G$ is the Heisenberg group or the group $G_{n}$ introduced in Section \ref{ss2.1} is the heat kernel for  a Schr\"{o}dinger operator $\mathcal{L}=-\Delta+V_{l}$ on $L^{2}\left( \mathbb{R}^{N}, dx \right)$  with a nonnegative polynomial potential $V_{l}$. In this case $V_{l}$ grows faster than $\vert x \vert^{\alpha}$ for some $\alpha$, then there are positive constants $c_{l}, C_{l}$ (depending on $\alpha$ \emph{etc}) such that the heat kernel for $\mathcal{L}$ satisfies

\[
k_{t}\left( x, y \right) \leqslant \frac{C_{l}}{t^{N/2}}\exp\left( -c_{l}t\left( \vert x \vert^{1+\alpha/2}+\vert y \vert^{1+\alpha/2}\right)\right)
\]
for all $x, y \in \mathbb{R}^{N}$ and $0<t\leqslant 1$.
\end{remark}

\begin{example}[Heisenberg group]\label{ex.4.1}  Let $H_{3}$ be the Heisenberg group identified with $\mathbb{R}^{3}$ with the multiplication given by
\[
(a , b , c)\cdot(a^{\prime}, b^{\prime}, c^{\prime}):=(a+a^{\prime}, b+b^{\prime}, c+c^{\prime}+\frac{1}{2}(a b^{\prime}-b a^{\prime})),
\]
and a Haar measure on $H_3$ then is the Lebesgue measure $dadbdc$ on $\mathbb{R}^3$.
 Let $\{X,Y,Z\}$ be the basis of the Lie algebra $\mathfrak{h}$ with the only non-zero bracket $[ X, Y ]=Z$. A polarizing (non-unique) sub-algebra for an element $l=\lambda Z^*$ can be chosen as $\mathfrak{m}=\operatorname{Span} \{ Y, Z \}$. Thus the corresponding subgroup is

 \[
 M=\exp{\mathfrak{m}}=\{\exp(bY+cZ),\ b,c \in \mathbb{R}\}
 \]
 identified with $\mathbb{R}^2$. The dual space of $H_3$ is given by
 \[
 \widehat{H}_3=\{\pi_{\lambda}, \ \lambda\in\mathbb{R}\}
 \]
where
\begin{eqnarray*}
\pi_{\lambda}\left( a, b, c\right) : L^2(\mathbb{R},\mathbb{C})& \longrightarrow & L^2(\mathbb{R},\mathbb{C})\\
f(x) &\longmapsto & e^{ 2\pi i \lambda(c+x b+\frac{ab}{2})}f(x+a).
\end{eqnarray*}
Then for any $f \in L^2(\mathbb{R},\mathbb{C})$

\begin{align*}
& \left( d\pi_{\lambda}\left( X \right)\right)f\left( x \right)= f^{\prime}\left( x \right),
\\
& \left( d\pi_{\lambda}\left( Y \right)\right)f\left( x \right)= 2\pi i \lambda x f\left( x \right),
\\
& \left( d\pi_{\lambda}\left( Z \right)\right)f\left( x \right)= 2\pi  i \lambda  f\left( x \right).
\end{align*}
The Plancherel measure on $\widehat{H}_3$ is $dP(\lambda)=|\lambda|d\lambda/4\pi^{2}$, where $d\lambda$ is the Lebesgue measure on $\mathbb{R}$. By using the ingredients above we are able to define the Fourier transform $\widehat{f}(\lambda)$ of a function $f\in L^2(H_3,\mathbb{C})$ as an operator on $L^2(\mathbb{R},\mathbb{C})$
\begin{align}
\left( \widehat{f}(\lambda)h\right)(x) &=\int_{\mathbb{R}^{3}} f(a,b,c)\left(\pi_{\lambda}(a,b,c)h\right)(x)dadbdc \label{e.3.19}
\\
&=\int_{\mathbb{R}^{3}} f(a,b,c)e^{ 2\pi i \lambda(-c-x b+\frac{ab}{2})}h(x-a)dadbdc, \hskip0.1in h \in L^2(\mathbb{R},\mathbb{C}). \notag
\end{align}
We can define a sub-Riemannian structure on $H_3$ by considering the two left invariant vector fields $\widetilde{X}=X(g), \widetilde{Y}=Y(g),\ g\in H_3$. Then the horizontal distribution is given by
\[
\mathcal{H}=\operatorname{Span} \{ \widetilde{X}, \widetilde{Y}\}
\]
and therefore the corresponding sub-Laplacian is
\[
\Delta_H = \widetilde{X}^2+\widetilde{Y}^2.
\]
The the Fourier transform $\Delta_H$ defined by \eqref{e.4.9} is

\[
\widehat{\Delta_{H}} \left( \pi_{\lambda} \right)f\left( x \right)=f^{\prime\prime}\left( x \right)-4\pi^2\lambda^{2}x^{2}f\left( x \right)=f^{\prime\prime}\left( x \right)-V_{\lambda}\left( x \right)f\left( x \right).
\]
Note that by using the global coordinates $(a,b,c)$ the sub-Laplacian is given by
\begin{align*}
\Delta_H f(a,b,c)&= \left(\widetilde{X}^2+\widetilde{Y}^2\right)f(a,b,c)\\
&=\left(\left(\partial_{a}-\frac{b}{2}\partial_{c}\right)^2+\left(\partial_{b}+\frac{a}{2}\partial_{c}\right)^2\right)f(a,b,c).
\end{align*}
 For $x, y \in \mathbb{R}$ we have
 \begin{align*}
 \gamma^{-1}(x)m\gamma(y) &=\exp(-xX)\exp(bY+cZ)\exp(yX)\\
&=(y-x, b, c-\frac{b}{2}(x+y)).
\end{align*}
The reduced heat kernel is then

\[
 k_t(x,y) =\int_M\chi_l(m)p_t(\gamma^{-1}(x)m\gamma(y))dm=\int_{\mathbb{R}^{2}} e^{ 2\pi i \lambda c} p_t(y-x, b, c-\frac{b}{2}(x+y))db dc.
\]
Observe that by \eqref{e.3.19} applied to $p_{t}$ and $h$ being the Dirac $\delta$ function (which can be made rigorous similarly to the proof of Theorem \ref{t.4.11}) we have

\begin{align*}
\left( \widehat{p_{t}}(\lambda)\delta\right)(x) &=\int_{\mathbb{R}^{3}} p_{t}(a,b,c)\left(\pi_{\lambda}(a,b,c)\delta\right)(x)dadbdc
\\
&=\int_{\mathbb{R}^{3}} p_{t}(a,b,c)e^{ 2\pi i \lambda(-c-x b+\frac{ab}{2})}\delta(x-a)dadbdc
\\
&=\int_{\mathbb{R}^{2}} p_{t}(x,b,c)e^{ 2\pi i \lambda(-c-\frac{xb}{2})}dbdc
\end{align*}
It is equal to the reduced heat kernel $k_{t}\left( 0, x \right)$ and taking the derivative in $t$ and  using the fact that $Y \in \mathfrak{m}$, and therefore $\widetilde{Y}$ is a skew-symmetric on $L^{2}\left( M, dm \right)$, we see that
\[
\partial_{t}k_t(0,x)=\widehat{\Delta_{H}} \left( \pi_{\lambda} \right)k_t(0,x)=\left( \frac{d^{2}}{dx^{2}}-V_{\lambda}\left( x \right)\right)k_t(0,x).
\]
Finally, in this case we can make \eqref{e.3.20} explicit. Namely, using the ingredients we described earlier, we see that the hypoelliptic heat kernel on $H_{3}$ is given by

\[
p_{t}\left( a, b, c \right)=\int\limits_{\mathbb{R}^{\times}} \int\limits_{ \mathbb{R}}
e^{2\pi i l\left( c +\frac{xb}{2} \right)}
k_{t}^{l}\left(  a+x, x \right)dx \vert l \vert dl,
\]
where $k_{t}^{l}\left( x, y \right)$ is the fundamental solution of Schr\"{o}dinger's equation with the generator

\[
\left( H_{l}f\right)\left( x \right)=\frac{d^{2}f}{dx^{2}}-l^{2} x^{2}f.
\]
In this case this heat kernel can be found explicitly, namely,

\begin{align*}
& k_{t}^{l}\left( x, y \right)=\left( \frac{l}{2\pi \sinh \left( 2lt \right)}\right)^{1/2}e^{-s_{t}^{l}\left( x, y \right)}, \text{ where }
\\
& s_{t}^{l}\left( x, y \right)=l\left( \left( x^{2}+y^{2} \right) \coth\left( 2lt \right) -\frac{xy}{2\sinh \left( 2lt \right)}\right)
\end{align*}
Thus we can use this form for $k_{t}^{l}$ to write an explicit expression for $p_{t}$

\[
p_{t}\left( a, b, c \right)=\int\limits_{0}^{\infty} \int\limits_{ \mathbb{R}}
2 l \cos \left(2\pi  l\left( c +\frac{xb}{2} \right)\right)
k_{t}^{l}\left(  a+x, x \right)dx  dl.
\]

\end{example}

\section{Hypoelliptic heat kernel on an $n$-step Lie group}\label{s4}

In this section we define a sub-Riemannian structure on the n-step Lie group $G_{n+1}$ described in Section \ref{ss2.1}, and then use the representations of $G_{n+1}$ to find an explicit expression for the corresponding hypoelliptic kernel.

In particular, we can use the matrix presentation \eqref{e.3.1} of the group $G_{n+1}$. We now introduce the isomorphism $\phi$ between $G_{n+1}$ and $\mathbb{R}^{n+1}$  by
\[
\phi(g)=(a,\mathbf{z}), \text{ where } \mathbf{z}:=\left( z_1,\dots, z_n \right),
\]
where $z$ is defined by \eqref{e.3.18}.
This isomorphism is a group isomorphism when $\mathbb{R}^{n+1}$ is endowed with the following product
\[
\left( a,\mathbf{z})\cdot(a',\mathbf{z}'\right):= \left( a+a',z_1+z_1',\dots,z_n+\sum_{i=0}^{n-1}\frac{a^i}{i!}z_{n-i}\right)
\]
Now, let us define a left-invariant sub-Riemannian structure on $G_{n+1}$ as presented in Section \ref{s3}. Consider two left invariant vector fields $X_1, X_2$ corresponding to $X, Y_1 \in \mathfrak{g}_{n+1}$ defined by \eqref{e.4.20}, and let
\begin{align*}
& \mathcal{H}(g):=\operatorname{Span}\{ X_1(g), X_2(g)\},
\\
& \mathbf{g}_g(X_i(g), X_j(g))=\delta_{ij}.
\end{align*}

Writing the group $G_{n+1}$ in coordinates $(a,\mathbf{z})\in \mathbb{R}^{n+1}$, we have the following expression for the left-invariant vector fields  $\widetilde{X_{1}}$ and $\widetilde{X_{2}}$
\begin{align*}
& \widetilde{X_1} =\frac{\partial}{\partial a},
\\
& \widetilde{X_2}=\frac{\partial}{\partial z_1}+a\frac{\partial}{\partial z_2}+\frac{a^2}{2!}\frac{\partial}{\partial z_3 }+\cdots+\frac{a^{n-1}}{(n-1)!}\frac{\partial}{\partial z_n}.
\end{align*}
The corresponding hypoelliptic Laplacian $\Delta_H$  on $G_{n+1}$ is given by
\[
\Delta_Hf=\left(\widetilde{X_1}^2+\widetilde{X_2}^2\right)f, \hskip0.1in f \in C_{c}^{\infty}\left( G_{n+1} \right).
\]
Our goal is to find an integral formula for the hypoelliptic heat kernel $p_{t}\left( g \right)$ as defined in Definition \ref{d.4.6} where $g$ is identified with a point in $\mathbb{R}^{n+1}$.

\begin{notation}\label{n.4.1} In what follows we denote by $\mathbb{R}^{n-1}_{\ast}$ the set

\[
\mathbb{R}^{n-1}_{\ast}:=\left\{\mathbf{\Lambda}_{n-1}=\left( \lambda_{1}, ..., \lambda_{n-1} \right) \in \mathbb{R}^{n-1}: \lambda_{n-1}\not=0  \right\}.
\]
\end{notation}
\begin{theorem}\label{t.5.1} The hypoelliptic heat kernel on the group $G_{n+1}$ is given by
\begin{align*}
& p_t(a, \mathbf{z})=
\\
& \int\limits_{\mathbb{R}^{n-1}}\int\limits_\mathbb{R} e^{2\pi i \left(\sum\limits_{k=1}^{n-1}\lambda_k B_k\left( x, \mathbf{z}\right)\right)}k_{t}^{\mathbf{\Lambda}_{n-1}}\left( x+a, x \right )d x |\lambda_{n-1}|d\mathbf{\Lambda}_{n-1},
\end{align*}
where $k_{t}^{\mathbf{\Lambda}_{n-1}}\left( x, y \right )$ is the fundamental solution to a Schr\"{o}dinger equation with a polynomial potential and
\[
B_k \left( x,\mathbf{z}\right) :=\sum_{i=1}^k\frac{z_i}{(k-i)!}x^{k-i},   k=1,2,\dots, {n-1},
\]
and

\[
d\mathbf{\Lambda}_{n-1}:=d\lambda_{1}...d\lambda_{n-1}
\]
is the Lebesgue measure.
\end{theorem}
\begin{remark} Note that the integral should be taken over $\mathbb{R}^{n-1}_{\ast}$, but as the integrand is $0$ when $\lambda_{n-1}=0$, we can instead integrate over $\mathbb{R}^{n-1}$.
\end{remark}
\begin{proof}
This formula can be derived using Theorem \ref{t.4.11}. Recall that we described all unitary representations of $G_{n+1}$  in Section \ref{s2}, and now we can use these results to apply Theorem \ref{t.4.11}. For this purpose it is enough to consider representations in the support of the Plancherel measure $P$. Thus we identify
\[
\widehat{G}_{n+1}\cong\left\{\pi_{\mathbf{\Lambda}_{n-1}}:  \mathbf{\Lambda}_{n-1}\in \mathbb{R}^{n-1}_{\ast}\right\},
\]
where $\pi_{\mathbf{\Lambda}_{n-1}}$ is a unitary operator on $L^{2}\left( \mathbb{R}, \mathbb{C} \right)$ defined by
\begin{align*}
\pi_{ \mathbf{\Lambda}_{n-1}}:  L^{2}\left( \mathbb{R}, \mathbb{C} \right) & \longrightarrow  L^{2}\left( \mathbb{R}, \mathbb{C} \right),
\\
f\left( x \right)& \longmapsto  \left( \pi^{\mathbf{\Lambda}_{n-1}}(a,\mathbf{z})f\right)\left( x \right),
\\
\left( \pi_{\mathbf{\Lambda}_{n-1}}(a,\mathbf{z})f\right)\left( x \right)& :=e^{2\pi i \left(\sum\limits_{k=1}^{n-1}\lambda_k B_k\left( x \right)\right)}f\left( x+a \right),
\end{align*}
where polynomials $B_k\left( x \right)$ are defined by \eqref{e.3.16}. Recall that by \eqref{e.3.5} the Plancherel measure on $\widehat{G}_{n+1}$ is identified with  the following measure on $\mathbb{R}^{n-1}$

\[dP(\mathbf{\Lambda}_{n-1})=\vert \operatorname{Pf}\left( l \right)\vert dm_{T}\left( l\right)=|\lambda_{n-1}|d\mathbf{\Lambda}_{n-1}.
\]

Consider the representation $\pi_{\mathbf{\Lambda}_{n-1}}$ acting on the representation space $L^2(\mathbb{R},\mathbb{C})$.  Using Definition \ref{d.4.9} we consider $d\pi_{\mathbf{\Lambda}_{n-1}}\left( X_{i} \right)$, $i=1, 2$ which are operators  on $L^2(\mathbb{R},\mathbb{C})$, which can be found explicitly as follows. Let $\mathbf{\Lambda}_{n-1}\in \mathbb{R}^{n-1}_{\ast}$, then
\begin{align*}
\left[d\pi_{\mathbf{\Lambda}_{n-1}}\left( X_{1} \right) f\right](x)&=\frac{d}{dt}\bigg|_{t=0}\pi_{\mathbf{\Lambda}_{n-1}}(e^{tX_{1}})f\left( x \right)
\\
&=\frac{d}{dt}\bigg|_{{t=0}}\pi_{\mathbf{\Lambda}_{n-1}}(t,0,\dots,0)f\left( x \right)
\\
&= \frac{d}{dt}\bigg|_{{t=0}}f\left( t+x \right)=f^{\prime}\left( x \right),
\end{align*}

\begin{align*}
\left[d\pi_{ \mathbf{\Lambda}_{n-1}}\left( X_{2} \right)f\right]\left( x \right)&=\frac{d}{dt}\bigg|_{{t=0}}\pi_{\mathbf{\Lambda}_{n-1}}(e^{tX_{2}})f\left( x \right)
\\ &=\frac{d}{dt}\bigg|_{{t=0}}\pi_{\mathbf{\Lambda}_{n-1}}(0,t,\dots, 0)f\left( x \right)
\\
&=2\pi i \left(\lambda_{n-1}\frac{x^{n-1}}{(n-1)!}+\sum_{j=1}^{n-2}\lambda_j\frac{x^{j-1}}{(j-1)!}\right)f\left( x \right)
\end{align*}
Thus, under the GFT of the hypoelliptic Laplacian (as in Theorem \ref{t.4.10}) is
\[
\widehat{\Delta_{H} f}\left(\pi_{\lambda_{n-1},\mathbf{\Lambda}_{n-2}} \right)\left( x \right)=\frac{d^2f}{dx^2}-4\pi^2\left(\lambda_{n-1}\frac{x^{n-1}}{(n-1)!}+\sum_{j=1}^{n-2}\lambda_j\frac{x^{j-1}}{(j-1)!}\right)^2 f\left( x \right).
\]
This is a Schr\"{o}dinger operator with a polynomial potential
\[
\mathcal{L}_{\mathbf{\Lambda}_{n-1}}:=\frac{d^2f}{dx^2}-V_{ \mathbf{\Lambda}_{n-1}}\left( x \right)f\left( x \right),
\]
where

\begin{equation}\label{e.5.1}
 V_{\mathbf{\Lambda}_{n-1}}\left( x \right):=4\pi^2\left(\frac{\lambda_{n-1}}{(n-1)!}x^{n-1}+\sum_{j=1}^{n-2}\frac{\lambda_j}{(j-1)!} x^{j-1}\right)^2.
\end{equation}
Denote by $k_{t}^{\mathbf{\Lambda}_{n-1}}\left(  x, y \right)$ the fundamental solution for the operator $\partial_t -\mathcal{L}_{\mathbf{\Lambda}_{n-1}}$. Then applying Theorem \ref{t.4.11}, one gets the kernel of the hypoelliptic heat equation on $G_{n+1}$ as follows
\begin{align*}
& p_t(a, \mathbf{z})=
\\
& \int\limits_{\mathbb{R}^{n-1}}\int\limits_\mathbb{R} e^{2\pi i \left(\sum\limits_{k=1}^{n-1}\lambda_k B_k\left( x \right)\right)}k_{t}^{\mathbf{\Lambda}_{n-1}}\left(  x+a, x \right )d x |\lambda_{n-1}|d\mathbf{\Lambda}_{n-1}.
\end{align*}
\end{proof}


\section{Short-time behavior of the hypoelliptic heat kernel and the Trotter product formula}\label{s6}
In this section we recall the techniques used by S\'eguin and Mansouri in \cite{SeguinMansouri2012}, where they show how the Trotter product formula for perturbation of a semigroup can be combined with an explicit formula for the hypoelliptic heat kernel to study the short-time behaviour of this heat kernel.

The following version of the Trotter product formula is well-suited for our purposes. Suppose  $C$ is an operator that can be written as a sum of two operators $C=A+B$. Then we can relate the semigroups generated by $A$ and $B$ with the semigroup generated by $C$ as follows (see \cite[Corollary 5.8]{EngelNagelBook}).
\begin{theorem}[Trotter product formula]
Let $(T_{t})_{t\geqslant 0}$ and $(S_{t})_{t\geqslant 0}$ be strongly continuous semigroups on a Banach space $X$ satisfying the stability condition
\[
\Vert [T_{t/N}S_{t/n}]^{N}\Vert\leqslant Me^{wt},  \text{ for all }t\geqslant 0,  N\in \mathbb{N}
\]
for some constants $M\geqslant 1, w \in \mathbb{R}$. Consider the sum $A+B$ on $D := D(A)\cap D(B)$ of the generators $(A,D(A))$ of $(T _{t})_{t\geqslant 0}$ and $(B, D(B))$ of $(S_{t})_{t\geqslant 0}$, and assume that $D$ and $(\lambda_0- A-B)D$ are dense in $X$ for some $\lambda_0\geqslant w$. Then the closure of the sum of these two operators $C := A + B$ generates a strongly continuous semigroup $(U_{t})_{t\geqslant 0}$ given by the Trotter product formula
\begin{equation}\label{ut}
U_{t}x = \lim_{N\rightarrow \infty} [T_{t/N}S_{t/N}]^{N}x
\end{equation}
with uniform convergence for $t$ on compact intervals.
\end{theorem}

As  S\'eguin and Mansouri observed, in most of cases the expression $[T(t/N)S(t/N)]^{N}$ is too complicated to be described explicitly. But in the special case when the operators satisfy some additional conditions and the Banach space $X$ is a nice function space such as $L^{p}, \leqslant p < \infty$, (see \cite[Section 3.3]{SeguinMansouri2012} for details), \eqref{ut} takes the form
\[
U_{t}f = T_{t}f +\lim_{N\rightarrow \infty}\frac{t}{N}\left(\sum_{k=0}^{N-1}T_{t/N}^{N-k}BT_{t/N}^{k}\right)f+\mathcal{O}(t^2)f,
\]
where  $\mathcal{O}(t^2)$ is an operator $D_t$ acting on $f$ such that $\Vert D_tf\Vert_{L^{1}}\leqslant Mt^2\Vert f\Vert_{L^{1}}$ for a constant $M$ and for all $t$ small enough.

Note that their main result \cite[Theorem 1]{SeguinMansouri2012} uses a number of assumptions formulated on \cite[p.3904]{SeguinMansouri2012}. Some of them are ambiguous (such as Assumption 1 which does not address the issue of the domains of unbounded operators), therefore we would like to use more concrete decompositions we have as a result of using Kirillov's orbit method. Namely, we consider the semigroup with the generator $\widehat{\Delta_{H}}\left( \pi_{l, \mathfrak{m}}\right)$ on $L^{2}\left( \mathbb{R}^{n-k}, dx \right)$ defined by \eqref{e.4.9}.  That is, we would like to write $\widehat{\Delta_{H}}\left( \pi_{l, \mathfrak{m}}\right)$ as a sum $A^{l}+B^{l}$ which are generators of the semigroups $\left\{T^{l}_{t}\right\}_{t\geqslant 0}$ and $\left\{S^{l}_{t}\right\}_{t\geqslant 0}$  respectively,  and which satisfy the following assumptions similar to \cite[Proposition 3]{SeguinMansouri2012}. Note that they base their analysis on the terminology of \cite{AgrachevBoscainGauthierRossi2009} which in particular leads to potential issues with \cite[Theorem 1]{SeguinMansouri2012}. But in our setting the direct integral terminology is not necessary, and in particular, all operators and their semigroups are defined on $L^{2}\left( \mathbb{R}^{n-k}, dx \right)$ , that is, on a space which does not depend on $\pi_{l, \mathfrak{m}}$. That is, in the case of nilpotent groups dependence on elements in the unitary dual $\widehat{G}$ is much easier to track.

\begin{enumerate}

\item  For each $t\geqslant 0$, $T^{l}_{t}$ is an integral operator with kernel $h_{t}^{l}( x, y)$
\[
\left( T^{l}_{t}f\right)(x)=\int_{\mathbb{R}^{n-k}}h_{t}^{l}( x, y) f(y)dy,
\]
\item There exists an integrable function $H_{t/N}^{l}(x,y)$, uniformly bounded in $x$ by an integrable function $G^{l}_{t}( x, y)$ for all $N\geqslant 1$, such that
\[
(t/N)\left(\left(T^{l}_{t/N}\right)^{N}B^{l}+...+T^{l}_{t/N}B^{l} \left(T^{l}_{t/N}\right)^{N-1}\right)f(x)
=\int_{\mathbb{R}^{n-k}}H^{l}_{t/N}(x,y) f(y)dy.
\]
\end{enumerate}
Then \cite[Theorem 1]{SeguinMansouri2012} can be interpreted as follows. The operator $\widehat{\Delta_{H}}\left( \pi_{l, \mathfrak{m}}\right)$ is the generator of the semigroup $e^{t\widehat{\Delta_{H}}\left( \pi_{l, \mathfrak{m}}\right)}$ which is an integral operator on $L^{2}\left( \mathbb{R}^{n-k}, dx \right)$  which can be written as the following expression for all small enough $t$

\begin{align}
& \left(e^{t \widehat{\Delta_{H}}\left( \pi_{l, \mathfrak{m}}\right)}f\right)(x)=\label{useful}
\\
&\int_{\mathbb{R}^{n-k}}\left(h_{t}^{l}( x, y)+t \lim_{N\rightarrow \infty}H_{t/N}^{l}( x, y)\right)f(y)dy+\left( \mathcal{O}(t^{3/2})f \right)(x), \notag
\end{align}
where the integral kernels $h_{t}^{l}$ and $H_{t}^{l}$ are defined above, and as before $\mathcal{O}(t^{3/2})f$ is an operator $D_t$ acting on $f$, such that $\Vert D_tf\Vert_{L^{1}} \leqslant Mt^2\Vert f\Vert_{L^{1}}$ for a constant $M$ and for all small enough $t$. Combining this with the reduced heat kernel introduced in Remark \ref{r.4.12}, we see that

\[
k^{l}_{t}\left( x, y \right)= h_{t}^{l}( x, y)+t \lim_{N\rightarrow \infty}H_{t/N}^{l}( x, y)+ \mathcal{O}(t^2).
\]
Recall that for the Heisenberg group $H_{3}$ and the group $G_{n+1}$

\[
\widehat{\Delta_{H}}\left( \pi_{l, \mathfrak{m}}\right)=\Delta-V_{l},
\]
where $V_{l}$ is a positive polynomial potential on $L^{2}\left( \mathbb{R}, dx \right)$. Let

\[
\left( A^{l}f \right)(x):=\Delta f(x)=f^{\prime \prime}(x)
\]
and $B^{l}$ be the multiplication operator
\[
\left( B^{l} f\right)\left( x \right):=V_{l}\left( x \right)f\left( x \right).
\]
We can write $V_{l}\left( x \right)$ as
\[
V_{l}\left( x \right)=-\left( a^{l}_{2m}x^{2m} +a^{l}_{2m-1}x^{2m-1} +\cdots+a^{l}_2x^2 +a^{l}_1x^1 +a^{l}_0 \right),
\]
where $m$ is the degree of the polynomial $V_{l}\left( x \right)$, and $a^{l}_{2m}$ and $a^{l}_0$ are positive constants. Then \cite[Corollary 2]{SeguinMansouri2012} is applicable, the equation \eqref{useful} becomes

\begin{align}
& \left( e^{t \widehat{\Delta_{H}}\left( \pi_{l, \mathfrak{m}}\right)}f \right)\left( x \right)= \label{e.6.3}
 \\
 & \int_\mathbb{R}\frac{1}{\sqrt{4\pi t}}e^{-\frac{(x-y)^2}{4t}}\left(1-a_0^{l} t-t\sum^{2m}_{k=1}\frac{a_k^{l}}{k+1}\sum_{i=0}^{k}x^iy^{k-i}\right)f(y)dy +\mathcal{O}(t^{3/2})f(x) \notag
\end{align}
for $f \in C_c^\infty(\mathbb{R})$. It is clear that for more general Schr\"{o}dinger-like operators in Theorem \ref{e.3.20} this approach can be used as well without making additional assumptions.


\subsection{Application to the $n$-step nilpotent Lie group $G_{n+1}$}
Recall that for $G_{n+1}$ the operator $\widehat{\Delta_{H}}$ is

\[
\left( \widehat{\Delta_{H}}\left( \pi_{l, \mathfrak{m}}\right)\right)f\left( x \right)=f^{\prime \prime }\left( x \right)-V_{l}\left( x \right)f\left( x \right), f \in C_{c}\left( G_{n+1}\right),
\]
where by \eqref{e.5.1} the potential is given by

\[
 V_{\mathbf{\Lambda}_{n-1}}\left( x \right)=4\pi^2\left(\frac{\lambda_{n-1}}{(n-1)!}x^{n-1}+\sum_{j=1}^{n-2}\frac{\lambda_j}{(j-1)!} x^{j-1}\right)^2.
\]
We  will use the following notation

 \[
 \mathbf{\Lambda}_k:= (\lambda_1,\lambda_2,\dots,\lambda_k).
 \]
Then as we mentioned previously
\begin{align*}
& A^{{\mathbf{\Lambda}_{n-1}}}f\left( x \right)=f^{\prime \prime }\left( x \right),
\\
&  \left( B^{{\mathbf{\Lambda}_{n-1}}}f \right)\left( x \right)=-\left(\lambda_n\frac{x^{n-1}}{(n-1)!}+\sum_{j=1}^{n-2}\lambda_j\frac{x^{j-1}}{(j-1)!}\right)^2f\left( x \right).
\end{align*}
Therefore the coefficients of the polynomial potential are
 \[
a_k^{{\mathbf{\Lambda}_{n-1}}}:=\sum^{k+1}_{i=1}\lambda_i\lambda_{k+2-i},  \  \  k=0,1,\dots,2(n-1).
\]
Finally, we use \eqref{e.6.3} to see that

\begin{align*}
& \left( e^{\widehat{\Delta_{H}}^{ {\mathbf{\Lambda}_{n-1}}}}f\right)\left( x \right) =
\\
& \int_\mathbb{R}\frac{1}{\sqrt{4\pi t}}e^{-\frac{(x-y)^2}{4t}}\left(1-\lambda_1^2 t-t\sum\limits_{k=1}^{2(n-1)}\frac{\sum\limits_{l=1}^{k+1}\lambda_l\lambda_{k+2-l}}{k+1} \sum\limits_{i=0}^{k}x^iy^{k-i}\right)f(y)dy
\\
& +\left(\mathcal{O}(t^{3/2})f\right)(x)
\end{align*}
with the reduced heat kernel being
\begin{align*}
&k^{{\mathbf{\Lambda}_{n-1}}}_t(x,y)=
\\
 &\frac{1}{\sqrt{4\pi t}}e^{-\frac{(x-y)^2}{4t}}\left(1-\lambda_1^2 t-t\sum_{k=1}^{2(n-1)}\frac{\sum_{l=1}^{k+1}\lambda_l\lambda_{k+2-l}}{k+1}\sum_{i=0}^{k}x^iy^{k-i}\right)+\mathcal{O}(t^{3/2})
\\
&=\frac{1}{\sqrt{4\pi t}}e^{-\frac{(x-y)^2}{4t}}\left(1-\lambda_1^2 t-t\sum_{k=1}^{2(n-1)}\frac{\sum_{l=1}^{k+1}\lambda_l\lambda_{k+2-l}}{k+1}\sum_{i=0}^{k}x^iy^{k-i}+\mathcal{O}(t^2)\right),
\end{align*}
where we used that $\mathcal{O}(t^{-1/2})\mathcal{O}(t^2) = \mathcal{O}(t^{3/2})$.

Now we can use Theorem \ref{t.5.1} \textcolor[rgb]{1.00,0.00,0.00}{ to see } that the hypoelliptic kernel is given by
\[
p_t(a, \mathbf{z})=\int\limits_{\mathbb{R}^{n-1}}\int\limits_\mathbb{R} e^{2\pi i \left(\sum\limits_{k=1}^{n-1}\lambda_k B_k\left( x ,\mathbf{z}\right)\right)}|\lambda_{n-1}|k_{t}^{\mathbf{\Lambda}_{n-1}}\left( x+a, x \right )d x d\mathbf{\Lambda}_{n-1},
\]
therefore
\begin{align*}
p_t(a,\mathbf{z})&=\frac{1}{\sqrt{4\pi t}}e^{-\frac{a^{2}}{4t}}\int_{\mathbb{R}^{n}}e^{2\pi i\left(\sum\limits_{k=1}^{n-1}\lambda_kB_k\right)}\times
\\
&\times\left[1-\lambda_1^2 t-t\sum_{k=1}^{2(n-1)}\frac{\sum\limits_{l=1}^{k+1}\lambda_l\lambda_{k+2-l}}{k+1}\sum_{i=0}^{k}(x+a)^ix^{k-i}+\mathcal{O}(t^2)\right]
\\
& dx|\lambda_{n-1}|d\mathbf{\Lambda}_{n-1}.
\end{align*}
Denote by $P_{2(n-1)}$ the polynomial of degree $2(n-1)$
\[
P_{2(n-1)}(x,a):=\lambda_1+\sum_{k=1}^{2(n-1)}\frac{\sum_{l=1}^{k+1}\lambda_l\lambda_{k+2-l}}{k+1}\sum_{i=0}^{k}(x+a)^ix^{k-i}.
\]
Choosing the substitution
\[
1-tP_{2(n-1)}(x,a)=e^{-tP_{2(n-1)}(x,a)}+\mathcal{O}(t^2),
\]
we get the following expression

\begin{align*}
&p_t(a,\mathbf{z})=\frac{1}{\sqrt{4\pi t}}e^{-\frac{a^2}{4t}} \times
\\
& \int_{\mathbb{R}^{n}}e^{2\pi i(\sum\limits_{k=1}^{n-1}\lambda_kB_k)}\left(e^{-tP_{2(n-1)}(x,a)}+\mathcal{O}(t^2)\right) dx|\lambda_{n-1}|d\lambda_{n-1}d\mathbf{\Lambda}_{n-2}=
\\
&\frac{1}{\sqrt{4\pi t}}e^{-\frac{a^2}{4t}}\int_{\mathbb{R}^{n}}e^{2\pi i(\sum\limits_{k=1}^{n-1}\lambda_kB_k)}e^{-tP_{2(n-1)}(x,a)}dx|\lambda_{n-1}|d\lambda_{n-1}d\mathbf{\Lambda}_{n-2}
\\
&+\frac{1}{\sqrt{4\pi t}}e^{-\frac{a^2}{4t}}\int_{\mathbb{R}^{n}}e^{2\pi i(\sum\limits_{k=1}^{n-1}\lambda_kB_k)}\mathcal{O}(t^2)dx|\lambda_{n-1}|d\lambda_{n-1}d\mathbf{\Lambda}_{n-2}.
\end{align*}
The first part is a $C^\infty$ function, and one can use special functions to find a more explicit form of this integral. The second integral is also a  $C^\infty$ function which is of order $\mathcal{O}(t^{3/2})$.

\providecommand{\bysame}{\leavevmode\hbox to3em{\hrulefill}\thinspace}
\providecommand{\MR}{\relax\ifhmode\unskip\space\fi MR }
\providecommand{\MRhref}[2]{%
  \href{http://www.ams.org/mathscinet-getitem?mr=#1}{#2}
}
\providecommand{\href}[2]{#2}

\end{document}